\documentclass[a4paper, 11pt, english]{amsart}
\usepackage[utf8]{inputenc}
\usepackage[T1]{fontenc}
\usepackage{graphicx}
\usepackage{stmaryrd}
\usepackage{tikz}
\usepackage{tikz-cd}
\usepackage[all]{xy}
\usepackage{comment}
\usepackage[a4paper]{geometry}
\usepackage{amsmath,amsfonts,amssymb,amsthm,epsfig,epstopdf,url,array}
\usepackage{rotating}
\usepackage{cleveref}
\usepackage{nameref}
\usepackage{enumitem}
\usepackage{comment}
\Crefname{paragraph}{Section}{Sections}
\usepackage{fancyhdr}

\newcommand{\ensemblenombre}[1]{\mathbb{#1}}
\newcommand{\N}{\ensemblenombre{N}}

\newcommand{\R}{} 
\renewcommand{\R}{\ensemblenombre{R}}

\newcommand{\norme}[1]{\left\lVert#1\right\rVert}

\renewcommand{\=}{\mathop{=} \limits}

\theoremstyle{plain} 
\newtheorem{prop}{Proposition}[section] 
\newtheorem{theo}[prop]{Theorem}

\newtheorem{lem}[prop]{Lemma}
\newtheorem{cor}[prop]{Corollary}
\theoremstyle{definition}
\newtheorem{defi}[prop]{Definition}
\newtheorem{rmk}[prop]{Remark}
\newtheorem{app}[prop]{Application}
\newtheorem{claim}[prop]{Fact}

\makeatletter
\let\original@addcontentsline\addcontentsline
\newcommand{\dummy@addcontentsline}[3]{}
\newcommand{\DeactivateToc}{\let\addcontentsline\dummy@addcontentsline}
\newcommand{\ActivateToc}{\let\addcontentsline\original@addcontentsline}
\makeatother

\begin{document}

\title{Global null-controllability and nonnegative-controllability of slightly superlinear heat equations}
\author{Kévin Le Balc'h}
\address{Kévin Le Balc'h, Univ Rennes, ENS Rennes, CNRS, IRMAR - UMR 6625, F-35000 Rennes, France}
\email{kevin.lebalch@ens-rennes.fr}
\maketitle
\begin{abstract}
We consider the semilinear heat equation posed on a smooth bounded domain $\Omega$ of $\R^{N}$ with Dirichlet or Neumann boundary conditions. The control input is a source term localized in some arbitrary nonempty open subset $\omega$ of $\Omega$. The goal of this paper is to prove the uniform large time global null-controllability for semilinearities $f(s) = \pm |s| \log^{\alpha}(2+|s|)$ where $\alpha \in [3/2,2)$ which is the case left open by Enrique Fernandez-Cara and Enrique Zuazua in 2000. It is worth mentioning that the free solution (without control) can blow-up. First, we establish the small-time global \textit{nonnegative-controllability} (respectively \textit{nonpositive-controllability}) of the system, i.e., one can steer any initial data to a nonnegative (respectively nonpositive) state in arbitrary time. In particular, one can act locally thanks to the control term in order to prevent the blow-up from happening. The proof relies on precise observability estimates for the linear heat equation with a bounded potential $a(t,x)$. More precisely, we show that observability holds with a sharp constant of the order $\exp\left(C \norme{a}_{\infty}^{1/2}\right)$ for \textit{nonnegative} initial data. This inequality comes from a new $L^1$ Carleman estimate. A Kakutani-Leray-Schauder's fixed point argument enables to go back to the semilinear heat equation. Secondly, the uniform large time null-controllability result comes from three ingredients: the global nonnegative-controllability, a comparison principle between the free solution and the solution to the underlying ordinary differential equation which provides the convergence of the free solution toward $0$ in $L^{\infty}(\Omega)$-norm, and the local null-controllability of the semilinear heat equation.
\end{abstract}
\small
\tableofcontents
\normalsize
\section{Introduction}

Let $T >0$, $N \in \N^{*}$, $\Omega$ be a bounded, connected, open subset of $\R^{N}$ of class $C^2$ and $n$ be the outer unit normal vector to $\partial \Omega$. We consider the semilinear heat equation with Neumann boundary conditions:
\begin{equation}
\label{heatSL}
\left\{
\begin{array}{l l}
\partial_t y-  \Delta y + f(y) =  h 1_{\omega} &\mathrm{in}\ (0,T)\times\Omega,\\
\frac{\partial y}{\partial n}= 0&\mathrm{on}\ (0,T)\times\partial\Omega,\\
y(0,.)=y_0& \mathrm{in}\ \Omega,
\end{array}
\right.
\end{equation}
where $f \in C^{1}(\R;\R)$.
\begin{rmk}
All our results stay valid for Dirichlet boundary conditions (see \Cref{SubDir}).
\end{rmk}
\indent In \eqref{heatSL}, $y=y(t,.) : \Omega \rightarrow \R$ is the \textit{state} to be controlled and $h=h(t,.) : \Omega \rightarrow \R$ is the \textit{control input} supported in $\omega$, a nonempty open subset of $\Omega$.\\
\indent We assume that $f$ satisfies
\begin{equation}
\label{fnUl}
f(0) = 0.
\end{equation}
In this case, $y = 0$ solves \eqref{heatSL} with $y_0 = 0$ and $h=0$.\\
\indent In the following, we will also assume that $f$ satisfies the restrictive growth condition
\begin{equation}
\label{finfty}
\exists \alpha > 0,\  \frac{f(s)}{|s| \log^{\alpha}(1+ |s|)} \rightarrow 0\ \text{as}\ |s| \rightarrow + \infty.
\end{equation}
Under the hypothesis \eqref{finfty}, blow-up may occur if $h=0$ in \eqref{heatSL}. Take for example $f(s) = -|s| \log^{\alpha}(1+|s|)$ with $\alpha >1$. The mathematical theory of blow-up for 
\begin{equation}
\label{heatSLBlow}
\left\{
\begin{array}{l l}
\partial_t y-  \Delta y = |y| \log^{\alpha}(1+|y|) &\mathrm{in}\ (0,T)\times\Omega,\\
y= 0&\mathrm{on}\ (0,T)\times\partial\Omega,\\
y(0,.)=y_0& \mathrm{in}\ \Omega,
\end{array}
\right.
\end{equation}
was established in \cite{GavaReg} and \cite{GavaHam}. It was shown that blow-up
\begin{itemize}
\item occurs \textit{globally} in the whole domain $\Omega$ if $\alpha < 2$, 
\item is of \textit{pointwise nature} if $\alpha >2$,
\item is ‘\textit{regional}’, i.e., it occurs in an open subset of $\Omega$ if $\alpha=2$.
\end{itemize}
See \cite[Section 2 and Section 5]{GaVaSurv} for a survey on this problem.\\
\indent The goal of this paper is to analyze the \textit{null-controllability} properties of \eqref{heatSL}.\\
\indent Let us define $Q_T := (0,T)\times \Omega$. We recall two classical definitions of null-controllability.
\begin{defi}
Let $T>0$. The system \eqref{heatSL} is
\begin{itemize}
\item  \textit{globally null-controllable} in time $T$ if for every $y_0 \in L^{\infty}(\Omega)$, there exists $h \in L^{\infty}(Q_T)$ such that the solution $y$ of \eqref{heatSL} satisfies $y(T,.)=0$.
\item \textit{locally null-controllable} in time $T$ if there exists $\delta_T>0$ such that for every $y_0 \in L^{\infty}(\Omega)$ verifying $\norme{y_0}_{L^{\infty}(\Omega)} \leq \delta_T$, there exists $h \in L^{\infty}(Q_T)$ such that the solution $y$ of \eqref{heatSL} satisfies $y(T,.)=0$.
\end{itemize}
\end{defi}
We have the following well-known local null-controllability result.
\begin{theo}
\label{ResLocal}
For every $T >0$, \eqref{heatSL} is locally null-controllable in time $T$.
\end{theo}
The proof of \Cref{ResLocal} is a consequence of the (global) null-controllability of the linear heat equation with a bounded potential (due to Andrei Fursikov and Oleg Imanuvilov, see \cite{FI} or \cite[Theorem 1.5]{FCG}) and the \textit{small $L^{\infty}$ perturbations method} (see \cite[Lemma 6]{AT} and \cite{AKBD}, \cite{B}, \cite{LB}, \cite{LCMML}, \cite{WZ} for other results in this direction).\\

\indent The following global null-controllability (positive) result has been proved independently by Enrique Fernandez-Cara, Enrique Zuazua (see \cite[Theorem 1.2]{FCZ}) and Viorel Barbu under a sign condition (see \cite[Theorem 2]{B2} or \cite[Theorem 3.6]{Barbu-Book}) for Dirichlet boundary conditions. It has been extended to semilinearities which can depend on the gradient of the state and to Robin boundary conditions (then to Neumann boundary conditions) by Enrique Fernandez-Cara, Manuel Gonzalez-Burgos, Sergio Guerrero and Jean-Pierre Puel in \cite{FCGBGP2} (see also \cite{DoFCGBZ} for the Dirichlet case).
\begin{theo}\label{posanswer}\cite[Theorem 1]{FCGBGP2}\\
We assume that \eqref{finfty} holds for $\alpha \leq 3/2$. Then, for every $T>0$, \eqref{heatSL} is globally null-controllable in time $T$.
\end{theo}
\begin{rmk}
Historically, the first global null-controllability (positive) result for \eqref{heatSL} with $f$ satisfying \eqref{finfty} was proved by Enrique Fernandez-Cara in \cite{FeCa} for $\alpha \leq 1$ and for Dirichlet boundary conditions.
\end{rmk}
The following global null-controllability (negative) result has been proved by Enrique Fernandez-Cara, Enrique Zuazua (see \cite{FCZ}).
\begin{theo}\label{neganswer}\cite[Theorem 1.1]{FCZ}\\
We set $f(s) := \int_{0}^{|s|} \log^{p}(1+\sigma) d\sigma$ with $p > 2$ and we assume that $\Omega \setminus \overline{\omega} \neq \emptyset$. Then, for every $T>0$, there exists an initial datum $y_0 \in L^{\infty}(\Omega)$ such that for every $h \in L^{\infty}(Q_T)$, the maximal solution $y$ of \eqref{heatSL} blows-up in time $T^{*} < T$.
\end{theo}
\begin{rmk}
Such a function $f$ does satisfy \eqref{finfty} for any $\alpha >p$ because $|f(s)| \sim |s| \log^{p}(1+|s|)$ as $|s| \rightarrow + \infty$. Then, \Cref{neganswer} shows that \eqref{heatSL} can fail to be null-controllable for every $T >0$ under the hypothesis \eqref{finfty} with $\alpha > 2$. \Cref{neganswer} comes from a localized estimate in $\Omega \setminus \overline{\omega}$ that shows that the control cannot compensate the blow-up phenomena occurring in $\Omega \setminus \overline{\omega}$ (see \cite[Section 2]{FCZ}).
\end{rmk}
\indent When the nonlinear term $f$ is dissipative, i.e., $s f (s) \geq 0$ for every $s \in \R$, then blow-up cannot occur. Furthermore, such a nonlinearity produces energy decay for the uncontrolled equation, therefore naively one may be led to believe that it can help in steering the solution to zero in arbitrary short time. The results of Sebastian Anita and Daniel Tataru show that this is false, more precisely that for ‘strongly’ superlinear $f$ one needs a sufficiently large time in order to bring the solution to zero. An intuitive explanation for this is that the nonlinearity is also damping the effect of the control as it expands from the controlled region into the uncontrolled region (see \cite{AT}).
\begin{theo}\label{AnTat}\cite[Theorem 3]{AT}\\
We set $f(s) := s\log^{p}(1+|s|)$ with $p>2$ and we assume that $\Omega \setminus \overline{\omega} \neq \emptyset$. Then, there exist $x_0 \in \Omega \setminus \overline{\omega}$, $T_0 \in (0,1)$ such that for every $T \in (0,T_0)$, $h\in L^{\infty}(Q_T)$, there exists $y_0 \in L^{\infty}(\Omega)$ such that the solution $y$ to \eqref{heatSL} satisfies $y(T,x_0) < 0$.
\end{theo}
\begin{rmk}
In particular, for such a $f$ as in \Cref{AnTat}, \eqref{heatSL} is not globally null-controllable in small time $T$. \Cref{AnTat} is due to pointwise upper bounds on the solution $y$ of \eqref{heatSL} which are independent of the control $h$ (see \cite[Section 3]{AT}).
\end{rmk}
\section{Main results}
\subsection{Small-time global nonnegative-controllability}
We introduce a new concept of controllability.
\begin{defi}
\label{DefiNegPos}
Let $T>0$. The system \eqref{heatSL} is \textit{globally nonnegative-controllable} (respectively \textit{globally nonpositive-controllable}) in time $T$ if for every $y_0 \in L^{\infty}(\Omega)$, there exists $h \in L^{\infty}(Q_T)$ such that the solution $y$ of \eqref{heatSL} satisfies
\begin{equation}
\label{yTNeg}
y(T,.)\geq 0 \qquad (\text{respectively}\ y(T,.) \leq 0).
\end{equation}
\end{defi}
The first main result of this paper is a small-time global nonnegative-controllability result for \eqref{heatSL}.
\begin{theo}
\label{TheoNeg}
We assume that \eqref{finfty} holds for $\alpha \leq 2$ and $f(s) \geq 0$ for $s \geq 0$ (respectively $f(s) \leq 0$ for $s \leq 0$). Then, for every $T>0$, \eqref{heatSL} is globally nonnegative-controllable (respectively globally nonpositive-controllable) in time $T$.
\end{theo}
\begin{rmk}
\Cref{TheoNeg} is almost sharp because it does not hold for $\alpha >2$ according to \Cref{AnTat}. The case where $|f(s)| \sim |s| \log^{2}(1+|s|)$ as $|s| \rightarrow +\infty$ is open.
\end{rmk}
\begin{rmk}
\Cref{TheoNeg} does not treat the case $f(s) = -s \log^{p}(1+|s|)$ with $p<2$ because of the sign condition.
\end{rmk}
\subsection{Large time global null-controllability}
The second main result of this paper is the following one.
\begin{theo}
\label{CorGlobal}
We assume that \eqref{finfty} holds for $\alpha \leq 2$, $f(s) > 0$ for $s > 0$ or $f(s) < 0$ for $s < 0$ and $1/f \in L^{1}([1,+\infty))$. Then, there exists $T$ sufficiently large such that \eqref{heatSL} is globally null-controllable in time $T$.
\end{theo}
\begin{rmk}
\Cref{CorGlobal} proves that \Cref{neganswer} is almost sharp. Indeed, let us take $f(s) = \int_{0}^{|s|} \log^{p}(1+\sigma) d\sigma$ with $p < 2$, then by \Cref{CorGlobal}, there exists $T$ sufficiently large such that \eqref{heatSL} is globally null-controllable in time $T$. In particular, one can find a localized control which prevents the blow-up from happening. The case $f(s) = \int_{0}^{|s|} \log^{2}(1+\sigma) d\sigma$ is open.
\end{rmk}
\begin{rmk}
\Cref{CorGlobal} does not treat the case $f(s) = -s \log^{p}(1+|s|)$ with $p<2$ because of the sign condition.
\end{rmk}
\begin{rmk}
The small-time global null-controllability of \eqref{heatSL} remains open when \eqref{finfty} holds for $3/2 < \alpha \leq 2$.
\end{rmk}
\subsection{Proof strategy of the small-time global nonnegative-controllability}
We will only prove the global nonnegative-controllability result. The nonpositive-controllability result is an easy adaptation.\\
\indent The proof strategy of \Cref{TheoNeg} will follow Enrique Fernandez-Cara and Enrique Zuazua's proof of \Cref{posanswer} (see \cite{FCZ}).\\
\indent The starting point is to get some precise \textit{observability estimates} for the linear heat equation with a bounded potential $a(t,x)$ for \textit{nonnegative initial data}. More precisely, we show that observability holds with a sharp constant of the order $\exp\left(C \norme{a}_{\infty}^{1/2}\right)$ for nonnegative initial data (see \Cref{TheoObsL^2L^1} below). This is done thanks to a \textit{new Carleman estimate in $L^1$} (see \Cref{CarlL1} below). This leads to a nonnegative-controllability result in $L^{\infty}$ in the linear case with an estimate of the \textit{control cost} of the order $\exp\left(C \norme{a}_{\infty}^{1/2}\right)$ which is the key point of the proof (see \Cref{TheoControlInfty} below).\\
\indent We end the proof of \Cref{TheoNeg} by a Kakutani-Leray-Schauder's fixed-point strategy. The idea of taking short control times to avoid blow-up phenomena is the same as in \cite{FCZ} and references therein. More precisely, the construction of the control follows two steps. The first step consists in steering the solution $y$ of \eqref{heatSL} to $y(T^{*},.) \geq 0$ in time $T^{*} \leq T$ with an appropriate choice of the control. Then, the two conditions: $f(0)=0$ and the dissipativity of $f$ in $\R^{+}$ imply that the free solution $y$ of \eqref{heatSL} (with $h = 0$) defined in $(T^{*}, T)$ stays nonnegative and bounded by using a comparison principle (see \Cref{Fixedpointsection}).
\subsection{Proof strategy of the large time global null-controllability}
We will only treat the case where $f(s) > 0$ for $s >0$. The other case, i.e., $f(s) < 0$ for $s < 0$ is an easy adaptation.\\
\indent The proof strategy of \Cref{CorGlobal} is divided into three steps.\\
\indent First, for every initial data $y_0 \in L^{\infty}(\Omega)$, one can steer the solution $y$ of \eqref{heatSL} in time $T_1 := 1$ (for instance) to a nonnegative state  by using \Cref{TheoNeg}.\\
\indent Secondly, we let evolve the system without control and we remark that 
$$ \forall (t,x) \in [T_1,+\infty)\times \Omega,\ 0 \leq y(t,x) \leq G(t),$$
with $G$ independent of $\norme{y(T_1,.)}_{L^{\infty}(\Omega)}$ and $G(t) \rightarrow 0$ when $t \rightarrow + \infty$. This kind of argument has already been used by Jean-Michel Coron in the context of the Burgers equation (see \cite[Theorem 8]{Coron-Op}).\\
\indent Finally, by using the second step, for $T_2$ sufficiently large, $y(T_2,.)$ belongs to a small ball of $L^{\infty}(\Omega)$ centered at $0$, where the local null-controllability holds (see \Cref{ResLocal}). Then, one can steer $y(T_2,.)$ to $0$ with an appropriate choice of the control.

\section{Parabolic equations: Well-posedness and regularity}
The goal of this section is to state well-posedness results, dissipativity in time in $L^p$-norm, maximum principle and $L^p$-$L^q$ estimates for linear parabolic equations. We also give the definition of a solution to the semilinear heat equation \eqref{heatSL}. The references of these results only treat the case of Dirichlet boundary conditions but the proofs can be easily adapted to Neumann boundary conditions.

\subsection{Well-posedness}
We introduce the functional space 
\begin{equation}
\label{defiYspaceL2}
W_T := L^2(0,T;H^1(\Omega)) \cap H^1(0,T;(H^{1}(\Omega))'),
\end{equation}
which satisfies the following embedding (see \cite[Section 5.9.2, Theorem 3]{E})
\begin{equation}
\label{injclassique}
W_T \hookrightarrow C([0,T];L^2(\Omega)).
\end{equation}
\subsubsection{Linear parabolic equations}
\begin{defi}
Let $a \in L^{\infty}(Q_T)$, $F \in L^2(Q_T)$ and $y_0 \in L^{2}(\Omega)$. A function $y \in W_T$ is a solution to 
\begin{equation}
\left\{
\begin{array}{l l}
\partial_t y  - \Delta y + a(t,x) y =  F &\mathrm{in}\ (0,T)\times\Omega,\\
\frac{\partial y}{\partial n}= 0&\mathrm{on}\ (0,T)\times\partial\Omega,\\
y(0,.)=y_0& \mathrm{in}\ \Omega,
\end{array}
\right.
\label{eqlinaBis}
\end{equation}
if for every  $w \in L^2(0,T;H^1(\Omega))$,
\begin{equation}
\int_0^T (\partial_t y ,w)_{((H^{1}(\Omega))',H^1(\Omega))} + \int_{Q_T}  \nabla y .  \nabla w + \int_{Q_T} a y w = \int_{Q_T} F w,
\label{formvar}
\end{equation}
and
\begin{equation}
y(0,.) = y_0 \ \mathrm{in}\ L^2(\Omega).
\label{condinitl2}
\end{equation}
\end{defi}
The following well-posedness result in $L^2$ holds for linear parabolic equations.
\begin{prop}\label{wpl2}
Let $a \in L^{\infty}(Q_T)$, $F \in L^2(Q_T)$ and $y_0 \in L^{2}(\Omega)$. The Cauchy problem \eqref{eqlinaBis} admits a unique weak solution $y \in W_T $.
Moreover, there \\exists $ C = C(\Omega)>0$ such that
\begin{equation}
 \norme{y}_{W_T} \leq C \exp\left(CT \norme{a}_{L^{\infty}(Q_T)}\right) \left(\norme{y_0}_{L^{2}(\Omega)}+\norme{F}_{L^{2}(Q_T)}\right).
 \label{estl2faible}
\end{equation}
\end{prop}
The proof of \Cref{wpl2} is based on \textit{Galerkin approximations, energy estimates and Gronwall's argument} (see \cite[Section 7.1.2]{E}).\\
\indent We also have the following classical $L^{\infty}$-estimate for \eqref{eqlinaBis}.
\begin{prop}\label{wplinfty}
Let $a \in L^{\infty}(Q_T)$, $F \in L^{\infty}(Q_T)$ and $y_0 \in L^{\infty}(\Omega)$. Then the solution $y$ of \eqref{eqlinaBis} belongs to $L^{\infty}(Q_T)$ and there exists $ C = C(\Omega)>0$ such that
\begin{equation}
\norme{y}_{L^{\infty}(Q_T)} \leq C \exp\left(CT \norme{a}_{L^{\infty}(Q_T)}\right) \left(\norme{y_0}_{L^{\infty}(\Omega)}+\norme{F}_{L^{\infty}(Q_T)}\right).
\label{estlinftyfaible}
\end{equation}
\end{prop}
The proof of \Cref{wplinfty} is based on \textit{Stampacchia's method} (see the proof of \cite[Chapter 3, Paragraph 7, Theorem 7.1]{LSU}).\\
\indent Let us also mention the dissipativity in time of the $L^p$-norm of the heat equation with a bounded potential.
\begin{prop}
\label{DissipProp}
Let $a \in L^{\infty}(Q_T)$, $y_0 \in L^{2}(\Omega)$ and $t_1 < t_2 \in [0,T]$. Then, there exists $C = C(\Omega)>0$ such that the solution $y \in W_T$ of \eqref{eqlinaBis} with $F=0$, satisfies for every $p \in [1,2]$,
\begin{equation}
\label{DissipEq}
\norme{y(t_2,.)}_{L^p(\Omega)} \leq C \exp\left(CT \norme{a}_{L^{\infty}(Q_T)}\right) \norme{y(t_1,.)}_{L^p(\Omega)}.
\end{equation}
\end{prop}
The proof of \Cref{DissipProp} is based on the application of the variational formulation \eqref{formvar} with a cut-off of $w = |y|^{p-2}y$ and a Gronwall's argument.
\subsubsection{Nonlinear parabolic equations}
We give the definition of a solution of \eqref{heatSL}.
\begin{defi}\label{defpropsolNL}
Let $y_0 \in L^{\infty}(\Omega)$, $h \in L^{\infty}(Q_T)$. A function $y \in W_T \cap L^{\infty}(Q_T)$ is the solution of \eqref{heatSL} if for every $w  \in L^2(0,T;H^1(\Omega))$,
\begin{align}
&\label{formvarNL1}  \int_0^T (\partial_t y ,w)_{((H^{1}(\Omega))',H^1(\Omega))} + \int_{Q_T}  \nabla y .  \nabla w  + \int_{Q_T} a y w =  \int_{Q_T} (f(y)+h1_{\omega}) w,
\end{align}
and 
\begin{equation}
\label{initNL}
y(0,.) = y_0\ \mathrm{in}\ L^{\infty}(\Omega).
\end{equation}
\end{defi}
The uniqueness of a solution to \eqref{heatSL} is an easy consequence of the fact that $f$ is locally Lipschitz because $f \in C^{1}(\R;\R)$.
\subsection{Maximum principle}
We state the maximum principle for the heat equation.
\begin{prop}
\label{propMaxPr}
Let $a \in L^{\infty}(Q_T)$, $F \leq G \in L^{2}(Q_T)$ and $y_0 \leq z_0 \in L^2(\Omega)$. Let $y$ and $z$ be the solutions to 
\small
\begin{equation}
\left\{
\begin{array}{ll}
\partial_t y  - \Delta y + a(t,x) y =  F,&\\
\frac{\partial y}{\partial n}=  0,&\\
y(0,.)=y_0,&
\end{array}
\right.
\left\{
\begin{array}{l l}
\partial_t z  - \Delta z + a(t,x) z =  G &\mathrm{in}\ (0,T)\times\Omega,\\
\frac{\partial z}{\partial n} = 0&\mathrm{on}\ (0,T)\times\partial\Omega,\\
z(0,.)=z_0& \mathrm{in}\ \Omega.
\end{array}
\right.
\label{CompyzProp}
\end{equation}
\normalsize
Then, we have the comparison principle
\begin{equation}
\label{comparaison}
\forall t \in [0,T],\ \text{a.e.}\ x \in \Omega,\ y(t,x) \leq z(t,x).
\end{equation}
\end{prop}
The proof of \Cref{propMaxPr} is based on the comparison principle for smooth solutions of \eqref{CompyzProp} (see \cite[Theorem 8.1.6]{WYW}) and a regularization argument.\\
\indent We state a comparison principle for the semilinear heat equation \eqref{heatSL} without control $h$.
\begin{prop}
\label{SubSuper}
Let $y_0 \in L^{\infty}(\Omega)$, $h=0$. We assume that there exist a subsolution $\underline{y}$ and a supersolution $\overline{y}$ in $L^{\infty}(Q_T)$ of \eqref{heatSL}, i.e., $\underline{y}$ (respectively $\overline{y}$) satisfies \eqref{formvarNL1}, \eqref{initNL} replacing the equality $=$ by the inequality $\leq$ (respectively by the inequality $\geq$). Moreover, we suppose that $\underline{y}$ and $\overline{y}$ are ordered in the following sense
$$\forall t \in [0,T],\ \text{a.e.}\ x \in \Omega,\ \underline{y}(t,x) \leq \overline{y}(t,x).$$
Then, there exists a (unique) solution $y$ of \eqref{heatSL}. Moreover, $y$ satisfies the comparison principle
\begin{equation}
\label{comparaisonNL}
\forall t \in [0,T],\ \text{a.e.}\ x \in \Omega,\ \underline{y}(t,x) \leq y(t,x) \leq \overline{y}(t,x).
\end{equation}
\end{prop}
For the proof of \Cref{SubSuper}, see \cite[Corollary 12.1.1]{WYW}.
\subsection{$L^p$-$L^q$ estimates}
We have the well-known regularizing effect of the heat semigroup.
\begin{prop}\label{L^pL^q}\cite[Proposition 3.5.7]{CaHa}\\
Let $1 \leq q \leq p \leq + \infty$, $y_0 \in L^2(\Omega)$ and $y$ be the solution to \eqref{eqlinaBis} with $(a,F)=(0,0)$. Then, there exists $C=C(\Omega,p,q)>0$ such that for every $t_1 < t_2 \in (0,T)$, we have
\begin{equation}
\label{EstiL^pL^q}
\norme{y(t_2,.)}_{L^p(\Omega)} \leq C(t_2-t_1)^{-\frac{N}{2}\left(\frac{1}{q}-\frac{1}{p}\right)}\norme{y(t_1,.)}_{L^q(\Omega)} 
\end{equation}
\end{prop}

\section{Global nonnegative-controllability of the linear heat equation with a bounded potential}
\label{GlobalPosSection}
\subsection{Statement of the result}
Let $a \in L^{\infty}(Q_T)$. We consider the heat equation with a bounded potential
\begin{equation}
\left\{
\begin{array}{l l}
\partial_t y - \Delta y + a(t,x) y = h 1_{\omega}  &\mathrm{in}\ (0,T)\times\Omega,\\
\frac{\partial y}{\partial n}= 0&\mathrm{on}\ (0,T)\times\partial\Omega,\\
y(0,.)=y_0& \mathrm{in}\ \Omega,
\end{array}
\right.
\label{eqlin}
\end{equation}
and the following adjoint equation
\begin{equation}
\left\{
\begin{array}{l l}
-\partial_t q - \Delta q + a(t,x) q = 0  &\mathrm{in}\ (0,T)\times\Omega,\\
\frac{\partial q}{\partial n}= 0 &\mathrm{on}\ (0,T)\times\partial\Omega,\\
q(T,.)=q_T& \mathrm{in}\ \Omega.
\end{array}
\right.
\label{eqlinaAdj}
\end{equation}
\indent The goal of this section is to prove the following theorem.
\begin{theo}
\label{TheoControlInfty}
For every $T>0$, \eqref{eqlin} is globally nonnegative-controllable in time $T$. More precisely, for every $T>0$, there exists $C = C(\Omega,\omega,T,a) > 0$, with 
\begin{align}
\label{coutcontrolereg}
&C(\Omega,\omega,T,a) = \exp\left(C(\Omega,\omega)\left(1+\frac{1}{T}+T\norme{a}_{L^{\infty}(Q_T)} + \norme{a}_{L^{\infty}(Q_T)}^{1/2}\right)\right)
\end{align} 
such that for every $y_0 \in L^2(\Omega)$, there exists $h \in L^{\infty}(Q_T)$ such that
\begin{equation}
\label{EstiControlinfty}
\norme{h}_{L^{\infty}(Q_T)} \leq C(\Omega,\omega,T,a) \norme{y_0}_{L^2(\Omega)},
\end{equation}
and
\begin{equation}
\label{yTNeqLin}
y(T,.) \geq 0.
\end{equation}
\end{theo}
\begin{rmk}
Actually, by looking carefully at the proof of \Cref{TheoControlInfty} (see \Cref{HumMethod} below), we can see that the control $h$ in \Cref{TheoControlInfty} can be chosen constant in the time and the space variables.
\end{rmk}
\begin{rmk}
It is well-known that \eqref{eqlin} is globally nonnegative-controllable in time $T$ because it is globally null-controllable in time $T$ (see \cite[Theorem 2]{FCGBGP}) but the most interesting point is the \textit{cost of nonnegative-controllability} given in \Cref{TheoControlInfty}. In particular, the exponent $1/2$ of the term $\norme{a}_{L^{\infty}(Q_T)}^{1/2}$ will be the key point to prove \Cref{TheoNeg} (see \Cref{Fixedpointsection}).
\end{rmk}

\subsection{A precise $L^2$-$L^1$ observability inequality for the linear heat equation with bounded potential and nonnegative initial data}
The proof of \Cref{TheoControlInfty} is a consequence of this kind of observability inequality.
\begin{theo}
\label{TheoObsL^2L^1}
For every $T >0$, there exists $C = C(\Omega,\omega,T,a) > 0$ of the form \eqref{coutcontrolereg} such that for every $q_T \in L^2(\Omega;\R^{+})$, the solution $q$ to \eqref{eqlinaAdj} satisfies 
\begin{equation}
\label{obsL^2L^1}
\norme{q(0,.)}_{L^2(\Omega)}^2 \leq C \left(\int_{0}^{T}\int_{\omega} q dx dt\right)^2.
\end{equation}
\end{theo}
An immediate corollary of \Cref{TheoObsL^2L^1} is this observability inequality $L^2$-$L^2$ that we state to discuss it below, but that will not be used in the
present article.
\begin{cor}
\label{TheoObsL^2L^2}
For every $T >0$, there exists $C = C(\Omega,\omega,T,a) > 0$ of the form \eqref{coutcontrolereg} such that for every $q_T \in L^2(\Omega;\R^{+})$ the solution $q$ to \eqref{eqlinaAdj} satisfies 
\begin{equation}
\label{obsL^2L^2}
\norme{q(0,.)}_{L^2(\Omega)}^2 \leq C \left(\int_{0}^{T}\int_{\omega} q^{2} dx dt\right).
\end{equation}
\end{cor}
It is well-known that null-controllability in $L^2$ is equivalent to an observability inequality in $L^2$ for every $q_T \in L^2(\Omega;\R)$ (see \cite[Theorem 2.44]{C}). The \textit{main idea} behind \Cref{TheoObsL^2L^2} is the fact that nonnegative-controllability in $L^2$ is a consequence of an observability inequality in $L^2$ for every $q_T \in L^2(\Omega;\R^{+})$ (see \Cref{HumMethod}).
\begin{rmk}
It is interesting to mention that \eqref{obsL^2L^2} holds with $C$ of the form 
\begin{align}
\label{coutcontrolereg2/3}
&C(\Omega,\omega,T,a) = \exp\left(C(\Omega,\omega)\left(1+\frac{1}{T}+T\norme{a}_{L^{\infty}(Q_T)} + \norme{a}_{L^{\infty}(Q_T)}^{2/3}\right)\right)
\end{align} 
for every $q_T \in L^2(\Omega;\R)$ (see \cite[Theorem 2]{FCGBGP}). The exponent $2/3$ of the term $\norme{a}_{L^{\infty}(Q_T)}^{2/3}$ is the key point to prove \Cref{posanswer}. Note that the optimality of the exponent $2/3$ has been proved by Thomas Duyckaerts, Xu Zhang and Enrique Zuazua in the context of parabolic systems in even space dimensions $N \geq 2$ and with Dirichlet boundary conditions (see \cite[Theorem 1.1]{DZZ} and also \cite[Theorem 5.2]{Z-hand} for the main arguments of the proof). \Cref{TheoObsL^2L^2} shows that we can actually decrease the exponent $2/3$ to the exponent $1/2$ for nonnegative initial data. In some sense, we can make the connection between the recent preprint of Camille Laurent and Matthieu Léautaud who disprove the Miller's conjecture about the short-time observability constant of the heat equation in the general case and show that the conjecture holds true for nonnegative initial data by using Li-Yau estimates (see \cite{Lale} and \cite{LiYau}).
\end{rmk}
\begin{rmk}
In the context of the wave equation in one space dimension, the (optimal) constant of observability inequality for the linear wave equation with a bounded potential is actually $\exp\left(C\left(1+\norme{a}_{L^{\infty}(Q_T)}^{1/2}\right)\right)$ (see \cite[Theorem 4]{Zu-wave}) which leads to the exact controllability of the semilinear wave equation in large time for semilinearities satisfying \eqref{finfty} with $\alpha <2$ (see \cite[Theorem 1]{Zu-wave} and also \cite[Problem 5.5]{Blondel} for the presentation of the related open problem in the multidimensional case). Roughly speaking, as an ordinary differential argument would indicate, this constant of observability inequality is very natural because the wave operator is of order two in the time and the space variables. Then, by analogy and by taking into account that the heat operator is of order one in the time variable and of order two in the space variable, one could rather expect a constant of obervability inequality of the order $\exp\left(C\norme{a}_{L^{\infty}(Q_T)}\right)$ or $\exp\left(C\norme{a}_{L^{\infty}(Q_T)}^{1/2}\right)$ which seem to be more intuitive than the term $\exp\left(C\norme{a}_{L^{\infty}(Q_T)}^{2/3}\right)$.
\end{rmk}
\subsection{A new $L^1$ Carleman estimate}\label{SectionCarlL1}
The goal of this section is to establish a $L^1$ Carleman estimate for nonnegative initial data (see \Cref{CarlL1} below). First, we introduce some classical weight functions for proving Carleman inequalities.
\begin{lem}
\label{Lempoids}
Let $\omega_0 \subset \subset \omega$ be a nonempty open subset. Then there exists $\eta^0 \in C^2(\overline{\Omega})$ such that $\eta^0 >0$ in $\Omega$, $\eta^0 = 0$ in $\partial \Omega$, and $|\nabla \eta^0| >0 $ in $\overline{\Omega \setminus \omega_0}$.
\end{lem}
A proof of this lemma can be found in \cite[Lemma 2.68]{C}.\\
\indent Let $\omega_0$ be a nonempty open set satisfying $\omega_0 \subset \subset \omega$ and let us set
\begin{equation}
\label{defalpha}
\alpha(t,x) := \frac{e^{2\lambda \norme{\eta^0}_{\infty}}-e^{\lambda\eta^{0}(x)}}{t(T-t)},
\end{equation}
\begin{equation}
\label{defxi}
\xi(t,x) := \frac{e^{\lambda\eta^{0}(x)}}{t(T-t)},
\end{equation}
for $(t,x) \in Q_T$, where $\eta^0$ is the function provided by \Cref{Lempoids} for this $\omega_0$ and $\lambda \geq 1$ is a parameter. \\
\indent We have the following new $L^1$ Carleman estimate.
\begin{theo}
\label{CarlL1}
There exist two constants $C:=C(\Omega,\omega)>0$ and $C_1 := C_1(\Omega,\omega)>0$, such that,
\begin{equation}
\forall \lambda \geq 1,\qquad \forall s \geq s_1(\lambda):=C(\Omega,\omega)e^{4\lambda \norme{\eta^0}_{\infty}}\left(T+T^{2}+T^2 \norme{a}_{L^{\infty}(Q_T)}^{1/2}\right),
\label{deflambd1s1}
\end{equation}
for every $q_T \in L^2(\Omega;\R^{+})$, the nonnegative solution $q$ of \eqref{eqlinaAdj} satisfies
\begin{align}
\label{carl}
\int_{Q_T} e^{-s\alpha}\xi^2 q dx dt\leq C_1\int_{(0,T)\times\omega} e^{-s\alpha} \xi^2 q  dxdt.
\end{align}
\end{theo}
\begin{proof} 
Unless otherwise specified, we denote by $C$ various positive constants varying from line to line which may depend on $\Omega$, $\omega$ but independent of the parameters $\lambda$ and $s$.\\
\indent We introduce other weights which are similar to $\alpha$ and $\xi$
\begin{equation}
\label{defalphat}
\widetilde{\alpha}(t,x) := \frac{e^{2\lambda \norme{\eta^0}_{\infty}}-e^{-\lambda\eta^{0}(x)}}{t(T-t)},
\end{equation}
\begin{equation}
\label{defxit}
\widetilde{\xi}(t,x) := \frac{e^{-\lambda\eta^{0}(x)}}{t(T-t)}.
\end{equation}
The following estimates
\begin{equation}
\label{UsefulEstimCarl}
\begin{array}{llll}
\displaystyle|\partial_i \alpha| = |-\partial_i \xi| &\leq C \lambda \xi,&|\partial_i \widetilde{\alpha}| =| -\partial_i \widetilde{\xi}| \displaystyle&\leq C \lambda \widetilde{\xi},\\
\displaystyle|\partial_t \alpha| &\leq  2 T  \xi^{2}e^{2\lambda \norme{\eta^0}_{\infty}},&\displaystyle |\partial_t \widetilde{\alpha}| &\leq  2 T  \widetilde{\xi}^{2}e^{4\lambda \norme{\eta^0}_{\infty}},\\
\displaystyle \xi (T/2)^{2} &\geq 1, & \widetilde{\xi} (T/2)^{2} &\geq e^{-\lambda \norme{\eta^0}_{\infty}},
\end{array}
\end{equation}
will be very useful for the proof.\\
\indent Let $q_T \in C_0^{\infty}(\Omega;\R^{+})$. The general case comes from an easy density argument by using the fact that $C_c^{\infty}(\Omega;\R^{+})$ is dense in $L^2(\Omega;\R^{+})$ for the $L^2(\Omega;\R)$ topology.\\
\indent The solution $q$ of \eqref{eqlinaAdj} is nonnegative by applying the maximum principle given in \Cref{propMaxPr} with $y=0$ and $z(t,x)=q(t-T,x)$.\\
\indent We define 
$$\psi := e^{-s\alpha} q\qquad \text{and}\qquad \widetilde{\psi} := e^{-s\widetilde{\alpha}} q.$$
\indent The proof is divided into five steps:
\begin{itemize}
\item \textbf{Step 1:} We integrate over $(0,T)\times \Omega$ an identity satisfied by $\psi$.
\item \textbf{Step 2:} We get an estimate which looks like to \eqref{carl} up to some boundary terms.
\item \textbf{Step 3:} We repeat the step 1 for $\widetilde{\psi}$.
\item \textbf{Step 4:} We repeat the step 2 for $\widetilde{\psi}$.
\item \textbf{Step 5:} We sum the estimates of the step 2 and the step 4 to get rid of the boundary terms.
\end{itemize}
\begin{rmk}
The ‘trick’ of the proof to get rid of the boundary terms is inspired by the proof of the usual $L^2$ Carleman estimate for Neumann boundary conditions due to Andrei Fursikov and Oleg Imanuvilov (see \cite[Chapter 1]{FI} and also \cite[Appendix]{FCGBGP}).
\end{rmk}
\indent \textbf{Step 1: An identity satisfied by $\psi$.} We readily obtain that
\begin{equation}
\label{eqM}
M \psi  = 0,
\end{equation}
where 
\begin{align}
\label{M}
\displaystyle M \psi & \displaystyle=-s \lambda^2  |\nabla \eta^0|^2 \xi \psi -2s \lambda \xi \nabla \eta^0 . \nabla \psi + \partial_t \psi\\
\displaystyle&\displaystyle\qquad + s^{2} \lambda^2 |\nabla \eta^0|^2 \xi^2 \psi + \Delta \psi + s \alpha_t \psi-a(t,x) \psi \notag\\
\displaystyle&\displaystyle\qquad - s \lambda \Delta \eta^0 \xi \psi.\notag
\end{align}
\begin{rmk}
The starting point, i.e., the identity \eqref{eqM} is the same as in the classical proof developed by Andrei Fursikov and Oleg Imanuvilov in \cite{FI} (see also \cite[Proof of Lemma 1.3]{FCG} or \cite[Section 7]{LLR}). But, from now, the proof strategy of the $L^1$-Carleman estimate is very different from the usual one of the $L^2$-Carleman estimate. Indeed, we will focus on the fourth right hand side term of \eqref{M}
$$ s^{2} \lambda^2 |\nabla \eta^0|^2 \xi^2 \psi.$$
It is nonnegative because $\psi$ is nonnegative and it is of order two in the parameter $s$ whereas the seventh right hand side term of \eqref{M}
$$ a(t,x) \psi,$$
is of order $0$ in the parameter $s$. This comparison suggests to integrate the identity \eqref{eqM} in order to obtain \eqref{carl} for $\lambda \geq 1$ and $s \geq s_1(\lambda)$ as defined in \eqref{deflambd1s1}.
\end{rmk}
\indent We integrate \eqref{eqM} over $(0,T)\times \Omega$
\begin{equation}
\label{Integrate}
\begin{array}{ll}
&\displaystyle\int_{Q_T} s^{2} \lambda^2 |\nabla \eta^0|^2 \xi^2 \psi - \int_{Q_T} 2s \lambda \xi \nabla \eta^0 . \nabla \psi  +\int_{Q_T}  \partial_t \psi +\int_{Q_T} \Delta \psi \\
&= \displaystyle \int_{Q_T} s \lambda^2  |\nabla \eta^0|^2 \xi \psi -   \int_{Q_T} s \alpha_t \psi+  \int_{Q_T}a(t,x) \psi \\
& \displaystyle \qquad +\int_{Q_T} s \lambda \Delta \eta^0 \xi \psi.
\end{array}
\end{equation}
Note that all the terms in \eqref{Integrate} are well-defined. Indeed, by using $q_T \in C_c^{\infty}(\Omega)$ and the parabolic regularity in $L^2$ to \eqref{eqlinaAdj} (see \cite[Theorem 2.1]{DHP}), we deduce that $q \in X_2 := L^2(0,T;H^2(\Omega))\cap H^1(0,T;L^2(\Omega))$ then $\psi \in X_2$.\\
\indent \textbf{Step 2: Estimates for $\psi$.} As a consequence of the properties of $\eta^{0}$ (see \Cref{Lempoids}), we have 
\begin{equation}
m := \min\left\{ |\nabla \eta^{0}(x)|^{2}\ ;\ x \in \overline{\Omega \setminus \omega_0}\right\}>0,
\label{defminm}
\end{equation}
which yields
\begin{align}
\label{SepInt}
&\int_{Q_T} s^{2} \lambda^2 |\nabla \eta^0|^2 \xi^2 \psi\\
& \geq\int_{(0,T)\times(\Omega\setminus\omega)} s^{2} \lambda^2 |\nabla \eta^0|^2 \xi^2 \psi \geq m \int_{Q_T} s^{2} \lambda^2 \xi^2 \psi - m \int_{(0,T)\times\omega} s^{2} \lambda^2  \xi^2 \psi. \notag
\end{align}
By combining \eqref{Integrate} and \eqref{SepInt}, we have 
\begin{equation}
\label{IntegrateBis}
\begin{array}{ll}
&\displaystyle m\int_{Q_T} s^{2} \lambda^2 \xi^2 \psi - \int_{Q_T} 2s \lambda \xi \nabla \eta^0 . \nabla \psi  +\int_{Q_T}  \partial_t \psi +\int_{Q_T} \Delta \psi \\
&\leq \displaystyle \int_{Q_T} s \lambda^2  |\nabla \eta^0|^2 \xi \psi + \int_{Q_T} s |\alpha_t| \psi+  \int_{Q_T}|a(t,x)| \psi \\
& \displaystyle \qquad +\int_{Q_T} s \lambda |\Delta \eta^0|\xi \psi + m \int_{(0,T)\times\omega} s^{2} \lambda^2 \xi^2 \psi.
\end{array}
\end{equation}
We have the following integration by parts
\begin{align}
\label{IntByParts1}
\begin{array}{ll}
\displaystyle-\int_{Q_T} 2s \lambda \xi \nabla \eta^0 . \nabla \psi =   \int_{Q_T} 2s \lambda\left(\nabla \xi . \nabla \eta^{0} \psi  + \xi \Delta \eta^{0} \psi  \right) -\int_{\Sigma_T} 2 s \lambda \xi \frac{\partial \eta^{0}}{\partial n} \psi d\sigma dt,
\end{array}
\end{align}
\begin{equation}
\int_{Q_T}  \partial_t \psi  =  \int_{\Omega} (\psi(T,.) - \psi(0,.) ) = 0,
\label{IntByParts2}
\end{equation}
\begin{equation}
\int_{Q_T} \Delta \psi = \int_{\Sigma_T}\frac{\partial\psi}{\partial n},
 \label{IntByParts3}
\end{equation}
where $\Sigma_T := (0,T)\times\partial\Omega$.\\
\indent From \eqref{IntegrateBis}, \eqref{IntByParts1}, \eqref{IntByParts2}, \eqref{IntByParts3}, we have
\begin{equation}
\label{IntegrateTri}
\begin{array}{ll}
&\displaystyle m \int_{Q_T} s^{2} \lambda^2 \xi^2 \psi -\int_{\Sigma_T} 2 s \lambda \xi \frac{\partial \eta^{0}}{\partial n} \psi+\int_{\Sigma_T}\frac{\partial\psi}{\partial n}\\
&\leq \displaystyle \int_{Q_T} s \lambda^2  |\nabla \eta^0|^2 \xi \psi + \int_{Q_T} s |\alpha_t| \psi+  \int_{Q_T}|a(t,x)| \psi \\
& \displaystyle \qquad +\int_{Q_T} 3 s \lambda |\Delta \eta^0|\xi \psi + \int_{Q_T} 2s \lambda|\nabla \xi| | \nabla \eta^{0}| \psi + m\int_{(0,T)\times\omega} s^{2} \lambda^2 \xi^2 \psi.
\end{array}
\end{equation}
\indent By using the first two lines of \eqref{UsefulEstimCarl} and $\lambda \geq 1$, we have 
\begin{align}
\begin{array}{ll}
\label{EstimUsefulBis}
&\displaystyle\int_{Q_T} s \lambda^2  |\nabla \eta^0|^2 \xi \psi + \int_{Q_T} s |\alpha_t| \psi+  \int_{Q_T}|a(t,x)| \psi\\
& \displaystyle+\int_{Q_T} 3 s \lambda |\Delta \eta^0|\xi \psi + \int_{Q_T} 2s \lambda|\nabla \xi| | \nabla \eta^{0}| \psi 
\\
& \displaystyle\quad\leq C \left( \int_{Q_T} s \lambda^2  \xi \psi + \int_{Q_T} s e^{2\lambda \norme{\eta^0}_{\infty}} T  \xi^{2} \psi +  \int_{Q_T} |a(t,x)| \psi +\int_{Q_T} s \lambda \xi \psi  \right)\\
& \displaystyle\quad\leq C \left( \int_{Q_T} s \lambda^2  \xi \psi + \int_{Q_T} s e^{2\lambda \norme{\eta^0}_{\infty}} T  \xi^{2} \psi +  \int_{Q_T} |a(t,x)| \psi  \right).
\end{array}
\end{align}
By combining \eqref{IntegrateTri} and \eqref{EstimUsefulBis}, we get
\begin{equation}
\label{IntegrateTriBis}
\begin{array}{ll}
&\displaystyle m \int_{Q_T} s^{2} \lambda^2 \xi^2 \psi -\int_{\Sigma_T} 2 s \lambda \xi \frac{\partial \eta^{0}}{\partial n} \psi+\int_{\Sigma_T}\frac{\partial\psi}{\partial n}\\
&\leq \displaystyle\quad\leq C \left( \int_{Q_T} s \lambda^2  \xi \psi + \int_{Q_T} s e^{2\lambda \norme{\eta^0}_{\infty}} T  \xi^{2} \psi +  \int_{Q_T} |a(t,x)| \psi  \right).
\end{array}
\end{equation}
\indent \textbf{Absorption.} The goal of this intermediate step is to absorb the right hand side of \eqref{IntegrateTriBis} by the first left hand side term of \eqref{IntegrateTriBis} by taking $s$ sufficiently large. In order to do this, it is useful to keep in mind the fact that $\lambda \geq 1$ and the third line of \eqref{UsefulEstimCarl} for the next estimates.\\
\indent  By taking $s \geq (T/2)^{2} (4C /m)$, we have $C s \xi \leq (m/4) (s \xi)^{2}$ and consequently
\begin{equation}
\label{Absorb1}
C \int_{Q_T} s \lambda^2  \xi \psi \leq \frac{m}{4} \int_{Q_T} s^{2} \lambda^2 \xi^2 \psi.
\end{equation}
By taking $s \geq T e^{2\lambda \norme{\eta^0}_{\infty}} (4C /m)$, we have $C s e^{2\lambda \norme{\eta^0}_{\infty}} T  \xi^{2} \leq (m/4)(\lambda s \xi)^{2}$ and consequently
\begin{equation}
\label{Absorb2}
C \int_{Q_T} s e^{2\lambda \norme{\eta^0}_{\infty}} T  \xi^{2} \psi \leq \frac{m}{4} \int_{Q_T} s^{2} \lambda^2 \xi^2 \psi.
\end{equation}
By taking $s \geq (T/2)^{2} \norme{a}_{L^{\infty}(Q_T)}^{1/2} (4C /m)^{1/2}$, we have $C \norme{a}_{L^{\infty}(Q_T)} \leq (m/4)(\lambda s \xi)^{2}$ and consequently
\begin{equation}
\label{Absorb3}
C \int_{Q_T} |a(t,x)| \psi \leq \frac{m}{4} \int_{Q_T} s^{2} \lambda^2 \xi^2 \psi.
\end{equation}
Therefore, by taking $s \geq s_1(\lambda)$ as defined in \eqref{deflambd1s1}, we have from \eqref{Absorb1}, \eqref{Absorb2} and \eqref{Absorb3} that 
\begin{equation}
\label{AbsorbTotal}
C \left( \int_{Q_T} s \lambda^2  \xi \psi + \int_{Q_T} s e^{2\lambda \norme{\eta^0}_{\infty}} T  \xi^{2} \psi +  \int_{Q_T} |a(t,x)| \psi\right) \leq \frac{3m}{4} \int_{Q_T} s^{2} \lambda^2 \xi^2 \psi.
\end{equation}
Then, from \eqref{IntegrateTriBis} and \eqref{AbsorbTotal}, for $s \geq s_1(\lambda)$, we get 
\begin{equation}
\label{IntegrateFour}
\displaystyle\frac{m}{4} \int_{Q_T} s^{2} \lambda^2 \xi^2 \psi -\int_{\Sigma_T} 2 s \lambda \xi \frac{\partial \eta^{0}}{\partial n} \psi +\int_{\Sigma_T}\frac{\partial\psi}{\partial n} \leq m \int_{(0,T)\times\omega} s^{2} \lambda^2 \xi^2 \psi.
\end{equation}
\indent \textbf{Step 3: An identity satisfied by $\widetilde{\psi}$.} We readily obtain that
\begin{equation}
\label{eqM2}
\widetilde{M}\widetilde{\psi} = 0,
\end{equation}
where 
\begin{align}
\label{M_2}
\widetilde{M} \widetilde{\psi} &=-s \lambda^2  |\nabla \eta^0|^2 \widetilde{\xi} \widetilde{\psi} +2s \lambda\widetilde{\xi} \nabla \eta^0 . \nabla \widetilde{\psi} + \partial_t \widetilde{\psi}\\
&\ + s^{2} \lambda^2 |\nabla \eta^0|^2 \widetilde{\xi}^2 \widetilde{\psi} + \Delta \widetilde{\psi}+ s \widetilde{\alpha}_t \widetilde{\psi} - a(t,x) \widetilde{\psi}\notag\\
&\ + s \lambda \Delta \eta^0 \widetilde{\xi} \widetilde{\psi}.\notag
\end{align}
\indent We integrate \eqref{eqM} over $(0,T)\times \Omega$
\begin{equation}
\label{IntegratePsiTilde}
\begin{array}{ll}
&\displaystyle\int_{Q_T} s^{2} \lambda^2 |\nabla \eta^0|^2 \widetilde{\xi}^2 \widetilde{\psi} + \int_{Q_T} 2s \lambda \widetilde{\xi} \nabla \eta^0 . \nabla \widetilde{\psi}  +\int_{Q_T}  \partial_t \widetilde{\psi} +\int_{Q_T} \Delta\widetilde{\psi} \\
&= \displaystyle \int_{Q_T} s \lambda^2  |\nabla \eta^0|^2 \widetilde{\xi} \widetilde{\psi} -   \int_{Q_T} s \widetilde{\alpha}_t \widetilde{\psi}+  \int_{Q_T}a(t,x) \widetilde{\psi} \\
& \displaystyle \qquad -\int_{Q_T} s \lambda \Delta \eta^0 \widetilde{\xi} \widetilde{\psi}.
\end{array}
\end{equation}
\indent \textbf{Step 4: Estimates for $\widetilde{\psi}$.}
By using \eqref{defminm}, we have
\begin{align}
\label{SepIntTilde}
&\int_{Q_T} s^{2} \lambda^2 |\nabla \eta^0|^2 \widetilde{\xi}^2 \widetilde{\psi}\\
& \geq\int_{(0,T)\times(\Omega\setminus\omega)} s^{2} \lambda^2 |\nabla \eta^0|^2 \widetilde{\xi}^2 \widetilde{\psi} \geq m \int_{Q_T} s^{2} \lambda^2 \widetilde{\xi}^2 \widetilde{\psi} - m \int_{(0,T)\times\omega} s^{2} \lambda^2  \widetilde{\xi}^2 \widetilde{\psi}. \notag
\end{align}
By combining \eqref{IntegratePsiTilde} and \eqref{SepIntTilde}, we have 
\begin{equation}
\label{IntegrateBisTilde}
\begin{array}{ll}
&\displaystyle m\int_{Q_T} s^{2} \lambda^2 \widetilde{\xi}^2 \widetilde{\psi} + \int_{Q_T} 2s \lambda \widetilde{\xi} \nabla \eta^0 . \nabla \widetilde{\psi}  +\int_{Q_T}  \partial_t \widetilde{\psi} +\int_{Q_T} \Delta \widetilde{\psi} \\
&\leq \displaystyle \int_{Q_T} s \lambda^2  |\nabla \eta^0|^2 \widetilde{\xi} \widetilde{\psi} + \int_{Q_T} s |\widetilde{\alpha}_t| \widetilde{\psi}+  \int_{Q_T}|a(t,x)| \widetilde{\psi} \\
& \displaystyle \qquad +\int_{Q_T} s \lambda |\Delta \eta^0|\widetilde{\xi} \widetilde{\psi} + m \int_{(0,T)\times\omega} s^{2} \lambda^2 \widetilde{\xi}^2 \widetilde{\psi}.
\end{array}
\end{equation}
We have the following integration by parts
\begin{align}
\label{IntByParts1Tilde}
\begin{array}{ll}
\displaystyle\int_{Q_T} 2s \lambda \widetilde{\xi} \nabla \eta^0 . \nabla \widetilde{\psi} =  - \int_{Q_T} 2s \lambda\left(\nabla \widetilde{\xi} . \nabla \eta^{0} \widetilde{\psi}  + \widetilde{\xi} \Delta \eta^{0} \widetilde{\psi}  \right) +\int_{\Sigma_T} 2 s \lambda \widetilde{\xi} \frac{\partial \eta^{0}}{\partial n} \widetilde{\psi},
\end{array}
\end{align}
\begin{equation}
\int_{Q_T}  \partial_t \widetilde{\psi}  =  \int_{\Omega} (\widetilde{\psi}(T,.) - \widetilde{\psi}(0,.) ) = 0,
\label{IntByParts2Tilde}
\end{equation}
\begin{equation}
\int_{Q_T} \Delta \widetilde{\psi} = \int_{\Sigma_T}\frac{\partial\widetilde{\psi}}{\partial n}.
 \label{IntByParts3Tilde}
\end{equation}
\indent From \eqref{IntegrateBisTilde}, \eqref{IntByParts1Tilde}, \eqref{IntByParts2Tilde}, \eqref{IntByParts3Tilde}, we have
\begin{equation}
\label{IntegrateTriTilde}
\begin{array}{ll}
&\displaystyle m \int_{Q_T} s^{2} \lambda^2 \widetilde{\xi}^2 \widetilde{\psi} +\int_{\Sigma_T} 2 s \lambda \widetilde{\xi} \frac{\partial \eta^{0}}{\partial n} \widetilde{\psi}+\int_{\Sigma_T}\frac{\partial\widetilde{\psi}}{\partial n}\\
&\leq \displaystyle \int_{Q_T} s \lambda^2  |\nabla \eta^0|^2 \widetilde{\xi} \widetilde{\psi} + \int_{Q_T} s |\widetilde{\alpha}_t| \widetilde{\psi}+  \int_{Q_T}|a(t,x)| \widetilde{\psi} \\
& \displaystyle \qquad +\int_{Q_T} 3 s \lambda |\Delta \eta^0|\widetilde{\xi} \widetilde{\psi} + \int_{Q_T} 2 s \lambda|\nabla \widetilde{\xi}| | \nabla \eta^{0}| \widetilde{\psi} + m\int_{(0,T)\times\omega} s^{2} \lambda^2 \widetilde{\xi}^2 \widetilde{\psi}.
\end{array}
\end{equation}
\indent By using the first two lines of \eqref{UsefulEstimCarl} and the fact that $\lambda \geq 1$, we have 
\begin{align}
\begin{array}{ll}
\label{EstimUsefulBisTilde}
&\displaystyle\int_{Q_T} s \lambda^2  |\nabla \eta^0|^2 \widetilde{\xi} \widetilde{\psi} + \int_{Q_T} s |\widetilde{\alpha}_t| \widetilde{\psi}+  \int_{Q_T}|a(t,x)| \widetilde{\psi}\\
& \displaystyle+\int_{Q_T} 3 s \lambda |\Delta \eta^0|\widetilde{\xi} \widetilde{\psi} + \int_{Q_T} 2s \lambda|\nabla \widetilde{\xi}| | \nabla \eta^{0}| \widetilde{\psi} 
\\
& \displaystyle\quad\leq C \left( \int_{Q_T} s \lambda^2  \widetilde{\xi} \widetilde{\psi} + \int_{Q_T} s e^{4\lambda \norme{\eta^0}_{\infty}} T  \widetilde{\xi}^{2} \widetilde{\psi} +  \int_{Q_T} |a(t,x)| \widetilde{\psi} +\int_{Q_T} s \lambda \widetilde{\xi} \widetilde{\psi}  \right)\\
& \displaystyle\quad\leq C \left( \int_{Q_T} s \lambda^2  \widetilde{\xi} \widetilde{\psi} + \int_{Q_T} s e^{4\lambda \norme{\eta^0}_{\infty}} T  \widetilde{\xi}^{2} \widetilde{\psi} +  \int_{Q_T} |a(t,x)| \widetilde{\psi} \right)
\end{array}
\end{align}
By combining \eqref{IntegrateTriTilde} and \eqref{EstimUsefulBisTilde}, we get
\begin{equation}
\label{IntegrateTriBisTilde}
\begin{array}{ll}
&\displaystyle m \int_{Q_T} s^{2} \lambda^2 \widetilde{\xi}^2 \widetilde{\psi} +\int_{\Sigma_T} 2 s \lambda \widetilde{\xi} \frac{\partial \eta^{0}}{\partial n} \widetilde{\psi}+\int_{\Sigma_T}\frac{\partial\widetilde{\psi}}{\partial n}\\
&\displaystyle\quad\leq C \left( \int_{Q_T} s \lambda^2  \widetilde{\xi} \widetilde{\psi} + \int_{Q_T} s e^{4\lambda \norme{\eta^0}_{\infty}} T  \widetilde{\xi}^{2} \widetilde{\psi} +  \int_{Q_T} |a(t,x)| \widetilde{\psi} \right)
\end{array}
\end{equation}
\indent \textbf{Absorption.} Note that we will use the third line of \eqref{UsefulEstimCarl} in the next four estimates. \\
\indent  By taking $s \geq e^{\lambda \norme{\eta^0}_{\infty}}(T/2)^{2} (4C /m)$, we have $C s \widetilde{\xi} \leq (m/4) (s \widetilde{\xi})^{2}$ and consequently
\begin{equation}
\label{Absorb1Tilde}
C \int_{Q_T} s \lambda^2  \widetilde{\xi} \widetilde{\psi} \leq \frac{m}{4} \int_{Q_T} s^{2} \lambda^2 \widetilde{\xi}^2 \widetilde{\psi}.
\end{equation}
By taking $s \geq T e^{4\lambda \norme{\eta^0}_{\infty}} (4C /m)$, we have $C s e^{2\lambda \norme{\eta^0}_{\infty}} T  \widetilde{\xi}^{2} \leq (m/4)(\lambda s \widetilde{\xi})^{2}$ and consequently
\begin{equation}
\label{Absorb2Tilde}
C \int_{Q_T} s e^{2\lambda \norme{\eta^0}_{\infty}} T  \widetilde{\xi}^{2} \widetilde{\psi} \leq \frac{m}{4} \int_{Q_T} s^{2} \lambda^2 \widetilde{\xi}^2 \widetilde{\psi}.
\end{equation}
By taking $s \geq e^{\lambda \norme{\eta^0}_{\infty}} (T/2)^{2} \norme{a}_{L^{\infty}(Q_T)}^{1/2} (4C /m)^{1/2}$, we have \\$C \norme{a}_{L^{\infty}(Q_T)} \leq (m/4)(\lambda s \widetilde{\xi})^{2}$ and consequently
\begin{equation}
\label{Absorb3Tilde}
C \int_{Q_T} |a(t,x)| \widetilde{\psi} \leq \frac{m}{4} \int_{Q_T} s^{2} \lambda^2 \widetilde{\xi}^2 \widetilde{\psi}.
\end{equation}
Therefore, by taking $s \geq s_1(\lambda)$ as defined in \eqref{deflambd1s1}, we have from \eqref{Absorb1}, \eqref{Absorb2} and \eqref{Absorb3Tilde} that 
\begin{equation}
\label{AbsorbTotalTilde}
C \left( \int_{Q_T} s \lambda^2  \widetilde{\xi} \widetilde{\psi} + \int_{Q_T} s e^{4\lambda \norme{\eta^0}_{\infty}} T  \widetilde{\xi}^{2} \widetilde{\psi} +  \int_{Q_T} |a(t,x)| \widetilde{\psi}  \right)\leq \frac{3m}{4} \int_{Q_T} s^{2} \lambda^2 \widetilde{\xi}^2 \widetilde{\psi}.
\end{equation}
Then, from \eqref{IntegrateTriBisTilde} and \eqref{AbsorbTotalTilde}, for $s \geq s_1(\lambda)$, we get 
\begin{equation}
\label{IntegrateFourPsitilde}
\displaystyle\frac{m}{4} \int_{Q_T} s^{2} \lambda^2 \widetilde{\xi}^2 \widetilde{\psi} +\int_{\Sigma_T} 2 s \lambda \widetilde{\xi} \frac{\partial \eta^{0}}{\partial n} \widetilde{\psi} +\int_{\Sigma_T}\frac{\partial\widetilde{\psi}}{\partial n} \leq m \int_{(0,T)\times\omega} s^{2} \lambda^2 \widetilde{\xi}^2 \widetilde{\psi}.
\end{equation}
\indent \textbf{Step 5: Elimination of the boundary terms.} From now, we take $s \geq s_1(\lambda)$. By summing \eqref{IntegrateFour} and \eqref{IntegrateFourPsitilde}, we get 
\begin{align}
\label{carlM1+M2}
\begin{array}{ll}
&\displaystyle\frac{m}{4}\int_{Q_T} s^{2} \lambda^2 \xi^2 \psi -\int_{\Sigma_T} 2 s \lambda \xi \frac{\partial \eta^{0}}{\partial n} \psi +\int_{\Sigma_T}\frac{\partial\psi}{\partial n} \\
&\displaystyle+ \frac{m}{4}\int_{Q_T} s^{2} \lambda^2 \widetilde{\xi}^2 \widetilde{\psi} +\int_{\Sigma_T} 2 s \lambda \widetilde{\xi} \frac{\partial \eta^{0}}{\partial n} \widetilde{\psi} +\int_{\Sigma_T}\frac{\partial\widetilde{\psi}}{\partial n}\\
& \displaystyle \ \leq m \left(\int_{(0,T)\times\omega} s^{2} \lambda^2 \xi^2 \psi+\int_{(0,T)\times\omega} s^{2} \lambda^2 \widetilde{\xi}^2 \widetilde{\psi}\right).
\end{array}
\end{align}
Since $\eta^0 = 0$ on $\partial \Omega$, we have
$$ \xi = \widetilde{\xi},\ \alpha = \widetilde{\alpha}\ \text{and}\ \psi = \widetilde{\psi}\qquad \text{on}\ \Sigma_T,$$
which leads to
\begin{equation}
\label{DeuxTermesNul}
 -\int_{\Sigma_T} 2 s \lambda \xi \frac{\partial \eta^{0}}{\partial n} \psi   +\int_{\Sigma_T} 2 s \lambda \widetilde{\xi}\frac{\partial \eta^{0}}{\partial n} \widetilde{\psi}  = 0.
\end{equation}
Moreover, we have 
$$ \partial_i \psi = e^{-s \alpha}(\partial_i q + s \lambda \partial_i \eta^0 \xi q),\ \partial_i \widetilde{\psi} = e^{-s \widetilde{\alpha}} ( \partial_i q - s \lambda \partial_i \eta^0 \widetilde{\xi} q),$$
whence by using $\frac{\partial q}{\partial n} = 0$ on $\Sigma_T$, we get
$$ \frac{\partial \psi}{\partial n} = s \lambda \frac{\partial \eta^{0}}{\partial n}  \xi e^{-s\alpha} q,\  \frac{\partial \widetilde{\psi}}{\partial n} = - s \lambda \frac{\partial \eta^{0}}{\partial n}  \widetilde{\xi} e^{-s\widetilde{\alpha}} q\qquad \textrm{on}\ \Sigma_T.$$
This leads to 
\begin{equation}
\label{Deuxautrestermesnuls}
\int_{\Sigma_T}\frac{\partial\psi}{\partial n}  + \int_{\Sigma_T}\frac{\partial\widetilde{\psi}}{\partial n} = 0.
\end{equation}
We get from \eqref{carlM1+M2}, \eqref{DeuxTermesNul} and \eqref{Deuxautrestermesnuls} \begin{align}
\label{carlM1+M2Bis}
\begin{array}{ll}
&\displaystyle\frac{m}{4}\left(\int_{Q_T} s^{2} \lambda^2 \xi^2 \psi + \int_{Q_T} s^{2} \lambda^2 \widetilde{\xi}^2 \widetilde{\psi}\right)\\
& \displaystyle \qquad\leq C \left(\int_{(0,T)\times\omega} s^{2} \lambda^2 \xi^2 \psi+\int_{(0,T)\times\omega} s^{2} \lambda^2 \widetilde{\xi}^2 \widetilde{\psi}\right).
\end{array}
\end{align}
By using the fact that $\widetilde{\xi} \leq \xi$, $e^{-s \widetilde{\alpha}} \leq e^{-s \alpha}$ in $Q_T$, we get from \eqref{carlM1+M2Bis} the Carleman estimate \eqref{carl}. This concludes the proof of \Cref{CarlL1}.
\end{proof}

\subsection{Proof of the $L^2$-$L^1$ observability inequality: \Cref{TheoObsL^2L^1}}
\label{SectionL^2L^1}
The goal of this subsection is to prove \Cref{TheoObsL^2L^1}, which is a consequence of \Cref{CarlL1}, $L^p$-$L^q$ estimates and the dissipativity in time of the $L^p$ norm of \eqref{eqlinaAdj}.
\begin{proof} \textbf{Step 1: $L^1$-$L^1$ observability inequality.} We fix $\lambda = 1$ and $s=s_1$ in \Cref{CarlL1} to get
\begin{equation}
\int_{Q_T} t^{-2}(T-t)^{-2} e^{-s\alpha} q dx dt\leq C_1(\Omega,\omega) \int_{(0,T)\times\omega} t^{-2}(T-t)^{-2} e^{-s\alpha}  q  dxdt.
\label{carlito}
\end{equation}
\indent First, we observe that in $(T/4,3T/4)\times\Omega$,
\begin{align}
\label{inT/43T/4}
\begin{array}{ll}
\displaystyle t^{-2}(T-t)^{-2} e^{-s\alpha}   & \displaystyle\geq  \frac{C}{T^4}\exp\left(-\frac{C(\Omega,\omega)\left(T+T^{2}+T^{2}\norme{a}_{L^{\infty}(Q_T)}^{1/2}\right)}{T^{2}}\right)\\
&\displaystyle \geq \frac{C}{T^4} e^{-C(\Omega,\omega)\left(1+\frac{1}{T}+\norme{a}_{L^{\infty}(Q_T)}^{1/2}\right)}.
\end{array}
\end{align}
Secondly, from the fact that $x^{2} e^{-Mx} \leq C/M^{2}$ for every $x, M \geq 0$ used with $x = t^{-1}(T-t)^{-1}$ and $M=C(\Omega,\omega)\left(T+T^{2}+T^{2}\norme{a}_{L^{\infty}(Q_T)}^{1/2}\right)$, we remark that in $(0,T)\times\omega$,
\begin{align}
\label{inomega}
\begin{array}{ll}
& \displaystyle t^{-2}(T-t)^{-2} e^{-s\alpha}  \\
&\displaystyle \leq t^{-2}(T-t)^{-2}\exp\left(-C(\Omega,\omega)\left(T+T^{2}+T^{2}\norme{a}_{L^{\infty}(Q_T)}^{1/2}\right)t^{-1}(T-t)^{-1}\right)\\
&\displaystyle \leq \frac{C}{\left(C(\Omega,\omega)\left(T+T^{2}+T^{2}\norme{a}_{L^{\infty}(Q_T)}^{1/2}\right)\right)^{2}}\\
&\displaystyle \leq  \frac{C(\Omega,\omega)}{T^{4}}.
\end{array}
\end{align}
Then, we get from \eqref{carlito}, \eqref{inT/43T/4} and \eqref{inomega}
\begin{equation}
\label{InobsL2L1Proof}
\int_{(T/4,3T/4)\times\Omega} q dx dt\leq e^{C(\Omega,\omega)\left(1+\frac{1}{T}+\norme{a}_{L^{\infty}(Q_T)}^{1/2}\right)} \int_{(0,T)\times\omega} q  dxdt.
\end{equation}
On the other hand, we obtain by the dissipativity in time of the $L^1$-norm (see \Cref{DissipProp} with $p=1$)
\begin{equation}
\label{Dissip}
\norme{q(T/4,.)}_{L^1(\Omega)} \leq \frac{2C\exp\left(CT \norme{a}_{L^{\infty}(Q_T)}\right)}{T} \int_{T/4}^{3T/4} \norme{q(t,.)}_{L^1(\Omega)} dt.
\end{equation}
By using \eqref{InobsL2L1Proof} and \eqref{Dissip}, we get 
\begin{equation}
\label{Firstestimate}
\norme{q(T/4,.)}_{L^1(\Omega)} \leq C(\Omega,\omega,T,a) \int_{(0,T)\times\omega} q  dxdt,
\end{equation}
where $C(\Omega,\omega,T,a)$ is defined in \eqref{coutcontrolereg}.\\
\indent From now, we denote by $C(\Omega,\omega,T,a)$ various positive constants varying from line to line which are of the form \eqref{coutcontrolereg}.\\
\indent \textbf{Step 2: Global $L^2$-$L^1$ estimate.} The goal of this step is to prove that
\begin{equation}
\label{L^2L^1Esti}
\norme{q(0,.)}_{L^2(\Omega)} \leq C(\Omega,\omega,T,a) \norme{q(T/4,.)}_{L^1(\Omega)}.
\end{equation}
To simplify the notations, we set $\widehat{q}(t) := q(T-t)$ for $t \in [0,T]$. Then, \eqref{L^2L^1Esti} rewrites as follows
\begin{equation}
\label{L^2L^1Estibis}
\norme{\widehat{q}(\widehat{T_2},.)}_{L^2(\Omega)} \leq C(\Omega,\omega,T,a) \norme{\widehat{q}(\widehat{T_1},.)}_{L^1(\Omega)}.
\end{equation}
with $\widehat{T_2} := T > \widehat{T_1} := 3T/4$.\\
\indent We introduce the following sequence
\begin{equation}
r_{0} := 1,\qquad \forall k \geq 0,\   r_{k+1} :=
\left\{
\begin{array}{c l}
\frac{N r_k}{N-r_k} & \mathrm{if}\  r_{k} < N,\\
2 r_k & \mathrm{if}\  r_{k} \geq N.
\end{array}
\right.
\label{defsequencep}
\end{equation}
We readily have from the definition \eqref{defsequencep} that
\begin{equation}
\label{CondIntegrability}
\forall k \geq 0,\ \beta_k:=\frac{N}{2}\left(\frac{1}{r_k}-\frac{1}{r_{k+1}}\right) \leq \frac{1}{2} < 1,
\end{equation}
and
\begin{equation}
\label{IntegerL}
\exists l \geq 1,\ r_l \geq 2.
\end{equation}
We also introduce a sequence of times 
\begin{equation}
\label{seqtimes}
\forall k \in \{0, \dots, l\},\ \tau_k := \widehat{T_1} + \frac{k}{l}(\widehat{T_2} - \widehat{T_1}).
\end{equation}
Let us remark that
\begin{equation}
\label{seqtimesBis}
\forall k \in \{0, \dots, l\},\ \tau_{k+1} - \tau_k = \frac{\widehat{T_2} - \widehat{T_1}}{l} = \frac{T}{2l}.
\end{equation}
By induction, we will show that
\begin{equation}
\label{induction}
\forall k \in \{0, \dots, l\},\ \norme{\widehat{q}(\tau_k,.)}_{L^{r_k}(\Omega)} \leq C(\Omega,\omega,T,a) \norme{\widehat{q}(\tau_0,.)}_{L^1(\Omega)}.
\end{equation}
The case $k=0$ is obvious (take $C_0=1$). Then, by denoting by $S(t) = e^{t\Delta}$ the heat-semigroup with Neumann boundary conditions, we have  for every $k \geq 0$,
\begin{equation}
\label{SemigroupChal}
\widehat{q}(\tau_{k+1}) = S(\tau_{k+1}-\tau_{k}) \widehat{q}(\tau_k) + \int_{\tau_k}^{\tau_{k+1}} S(\tau_{k+1}-s) (-a(s,.) \widehat{q}(s)) ds,
\end{equation}
from the equation satisfied by $\widehat{q}$ (see \eqref{eqlinaAdj}).\\
\indent We assume that \eqref{induction} holds for $k \in \{0, \dots, l\}$. From \eqref{SemigroupChal}, \eqref{CondIntegrability} and the regularizing effect $L^{r_k}$-$L^{r_{k+1}}$ of the heat-semigroup (see \Cref{L^pL^q}), we have
\begin{align}
\label{Unemajchiante}
\begin{array}{ll}
\displaystyle\norme{\widehat{q}(\tau_{k+1})}_{L^{r_{k+1}}(\Omega)}&\displaystyle\leq (\tau_{k+1}-\tau_{k})^{-\beta_k} \norme{\widehat{q}(\tau_{k})}_{L^{r_k}(\Omega)}\\
&\displaystyle\quad + \int_{\tau_{k}}^{\tau_{k+1}} (\tau_{k+1}-s)^{-\beta_k} \norme{a}_{L^{\infty}(Q_T)} \norme{\widehat{q}(s)}_{L^{r_k}(\Omega)} ds\\
& \leq A_{1,k} + A_{2,k},
\end{array}
\end{align}
where 
\begin{equation}
\label{MajA_{1,k}}
A_{1,k} :=(\tau_{k+1}-\tau_{k})^{-\beta_k} \norme{\widehat{q}(\tau_{k})}_{L^{r_k}(\Omega)},
\end{equation}
and 
\begin{equation}
\label{MajA_{2,k}}
A_{2,k} := \int_{\tau_{k}}^{\tau_{k+1}} (\tau_{k+1}-s)^{-\beta_k} \norme{a}_{L^{\infty}(Q_T)} \norme{\widehat{q}(s)}_{L^{r_k}(\Omega)} ds.
\end{equation}
From \eqref{MajA_{1,k}}, \eqref{seqtimesBis}, \eqref{CondIntegrability} and \eqref{induction}, we have
\begin{equation}
\label{MajA_{1,k}Bis}
A_{1,k} \leq C T^{-\beta_k} C(\Omega,\omega,T,a) \norme{\widehat{q}(\tau_0,.)}_{L^1(\Omega)} \leq C(\Omega,\omega,T,a) \norme{\widehat{q}(\tau_0,.)}_{L^1(\Omega)} .
\end{equation}
From \eqref{MajA_{2,k}}, the dissipativity in time of the $L^{r_k}$-norm (see \Cref{DissipProp}), the induction assumption \eqref{induction}, \eqref{CondIntegrability} and \eqref{seqtimesBis}, we have
\begin{equation}
\label{MajA_{2,k}Bis}
\begin{array}{ll}
A_{2,k}& \displaystyle\leq  \norme{a}_{\infty}\int_{\tau_{k}}^{\tau_{k+1}} (\tau_{k+1}-s)^{-\beta_k} C e^{CT \norme{a}_{\infty}} \norme{\widehat{q}(\tau_{k})}_{L^{r_k}(\Omega)} ds\\
& \displaystyle \leq  C \norme{a}_{\infty} e^{CT \norme{a}_{\infty}} C(\Omega,\omega,T,a) \norme{\widehat{q}(\tau_0,.)}_{L^1(\Omega)} (\tau_{k+1}-\tau_k)^{-\beta_k+1}\\
& \displaystyle \leq C(\Omega,\omega,T,a)  \norme{a}_{\infty} T^{-\beta_k+1}  \norme{\widehat{q}(\tau_0,.)}_{L^1(\Omega)} \\
&\displaystyle \leq C(\Omega,\omega,T,a)  \norme{a}_{\infty} (T+1)  \norme{\widehat{q}(\tau_0,.)}_{L^1(\Omega)} \\
&\displaystyle \leq C(\Omega,\omega,T,a) \left(e^{T \norme{a}_{\infty}}+2e^{\norme{a}_{\infty}^{1/2}}\right)\norme{\widehat{q}(\tau_0,.)}_{L^1(\Omega)}\\
& \leq C(\Omega,\omega,T,a) \norme{\widehat{q}(\tau_0,.)}_{L^1(\Omega)} .
\end{array}
\end{equation}
The estimates \eqref{Unemajchiante}, \eqref{MajA_{1,k}Bis} and \eqref{MajA_{2,k}Bis} prove \eqref{induction} for $(k+1)$ and concludes the induction. Thus, \eqref{induction} holds for $k=l$, which combined with \eqref{IntegerL} and \eqref{seqtimes}, yields \eqref{L^2L^1Estibis}.\\
\color{black}
\indent \textbf{Step 3:} By using \eqref{Firstestimate} and \eqref{L^2L^1Esti}, we prove \eqref{obsL^2L^1} and consequently \Cref{TheoObsL^2L^1}.
\end{proof}

\subsection{Proof of the linear global nonnegative-controllability: \Cref{TheoControlInfty}}
\label{HumMethod}
The goal of this section is to prove \Cref{TheoControlInfty}. The following proof is inspired by the so-called \textit{Hilbert Uniqueness method} due to Jacques-Louis Lions (see \cite{LJL} and more precisely \cite[Section 2.1]{Zu-Fi}).
\begin{proof}
The proof is divided into two steps. First, we build a sequence of controls\\ $h_{\varepsilon} \in L^{\infty}((0,T)\times\omega)$ with $\varepsilon > 0$ which provide the approximate nonnegative-controllability of \eqref{eqlin}. Secondly, we pass to the limit when $\varepsilon$ tends to $0$.\\
\indent \textbf{Step 1.} Let us fix $T > 0$, $a \in L^{\infty}(Q_T)$ and $y_0 \in L^2(\Omega)$. For any $\varepsilon \in(0,1)$, we consider the following functional: for every $q_T \in L^2(\Omega;\R^{+})$,
\begin{equation}
\label{Functional}
J_{\varepsilon}(q_T) = \frac{1}{2} \left(\int_{(0,T)\times\omega} q dx dt\right)^2 + \varepsilon \norme{q_T}_{L^2(\Omega)} + \int_{\Omega} q(0,x) y_0(x) dx,
\end{equation}
where $q$ is the solution to \eqref{eqlinaAdj}.\\
\indent The functional $J_{\varepsilon}$ is continuous, convex and coercive on the unbounded closed convex set $L^2(\Omega;\R^{+})$. More precisely, we will show that
\begin{equation}
\label{Coercivity}
\liminf_{\norme{q_T}_{L^2(\Omega)} \rightarrow + \infty} \frac{J_{\varepsilon}(q_T)}{\norme{q_T}_{L^2(\Omega)}} \geq \varepsilon.
\end{equation}
Indeed, given a sequence $(q_{T,k})_{k \geq 0} \in L^2(\Omega)$ with $\norme{q_{T,k}}_{L^2(\Omega)} \rightarrow + \infty$, we normalize it:
\begin{equation*}
\widetilde{q}_{T,k} := \frac{q_{T,k}}{\norme{q_{T,k}}_{L^2(\Omega)}},
\end{equation*}
and we denote by $\widetilde{q}_k$ the solution to \eqref{eqlinaAdj} associated to the initial data $\widetilde{q}_{T,k}$.
We have
\begin{equation}
\frac{J_{\varepsilon}(q_{T,k})}{\norme{q_{T,k}}_{L^2(\Omega)}} = \frac{\norme{q_{T,k}}_{L^2(\Omega)}}{2} \left(\int_{(0,T)\times\omega} \widetilde{q}_k dx dt\right)^2 + \varepsilon + \int_{\Omega} \widetilde{q}_k(0,x) y_0(x) dx.
\label{Jdivide}
\end{equation}
We distinguish the following two cases.\\
\indent \textbf{Case 1:}
\begin{equation}
\label{case1}
\liminf_{k \rightarrow + \infty} \int_{(0,T)\times\omega} \widetilde{q}_kdx dt > 0.
\end{equation}
When \eqref{case1} holds, we clearly have
\begin{equation*}
\liminf_{k \rightarrow + \infty}
\frac{J_{\varepsilon}(q_{T,k})}{\norme{q_{T,k}}_{L^2(\Omega)}} = +\infty \geq \varepsilon
\end{equation*}
\indent \textbf{Case 2:}
\begin{equation}
\liminf_{k \rightarrow + \infty} \int_{(0,T)\times\omega}\widetilde{q}_kdx dt = 0.
\label{linfinfNul}
\end{equation}
In this case, by using the estimate \eqref{estl2faible} of \Cref{wpl2}, the embedding \eqref{injclassique} and \eqref{linfinfNul}, extracting subsequences (that we denote by the index $k$ to simplify the notation), we deduce that there exists $\widetilde{q} \in W_T$ such that
\begin{equation}
\label{convfaible}
\widetilde{q}_k \rightharpoonup \widetilde{q}\ \textrm{in}\ W_T,
\end{equation}
\begin{equation}
\label{convfaiblePont}
\widetilde{q}_k(0,.) \rightharpoonup \widetilde{q}(0,.)\ \textrm{in}\ L^2(\Omega),
\end{equation}
\begin{equation}
\label{convfortesur}
\int_{(0,T)\times\omega} \widetilde{q}_k dx dt \rightarrow 0.
\end{equation}
By using Aubin Lions' lemma (see \cite[Section 8, Corollary 4]{S}) and \eqref{convfaible}, $(\widetilde{q}_k)_{k \in \N}$ is relatively compact in $L^2(Q_T)$, then up to a subsequence we have 
\begin{equation}
\label{convL^2}
\widetilde{q}_k \rightarrow \widetilde{q}\ \textrm{in}\ L^2(Q_T;\R^{+}).
\end{equation}
In view of \eqref{convfortesur} and \eqref{convL^2}, we have
\begin{equation}
\widetilde{q} = 0 \ \textrm{in}\ (0,T)\times \omega.
\label{Nulomega}
\end{equation}
Then, by using \eqref{Nulomega} and the observability inequality \eqref{obsL^2L^1}, we have 
\begin{equation}
\widetilde{q}(0,.)=0.
\label{NulCondFin}
\end{equation}
Consequently, by combining \eqref{convfaiblePont} and \eqref{NulCondFin}, we have 
$$ \int_{\Omega} \widetilde{q}_k(0,x) y_0(x) dx \rightarrow 0,$$
which yields \eqref{Coercivity} thanks to \eqref{Jdivide}.\\
\indent We deduce that $J_{\varepsilon}$ admits a minimum $q_{\varepsilon,T} \in L^2(\Omega;\R^{+})$. We take 
\begin{equation}
\label{PoseControl}
h_{\varepsilon} := \left(\int_{(0,T)\times\omega} q_{\varepsilon}\right) 1_{\omega},
\end{equation}
and we denote by $y_{\varepsilon} \in W_T \cap L^{\infty}(Q_T)$ the solution to 
\begin{equation}
\left\{
\begin{array}{l l}
\partial_t y_{\varepsilon} - \Delta y_{\varepsilon} + a(t,x) y_{\varepsilon} = h_{\varepsilon} 1_{\omega} &\mathrm{in}\ (0,T)\times\Omega,\\
\frac{\partial y_{\varepsilon}}{\partial n}= 0&\mathrm{on}\ (0,T)\times\partial\Omega,\\
y_{\varepsilon}(0,.)=y_0& \mathrm{in}\ \Omega.
\end{array}
\right.
\label{eqlinEps}
\end{equation}
We use the fact that $J_{\varepsilon}(q_{T,\varepsilon}) \leq J_{\varepsilon}(0) = 0$ to get
\begin{equation}
\label{EstimCOntrol}
\frac{1}{2}\left(\int_{(0,T)\times\omega} q_{\varepsilon}\right)^2 +  \varepsilon \norme{q_{\varepsilon,T}}_{L^2(\Omega)} \leq -\int_{\Omega} q_{\varepsilon}(0,x) y_0(x) dx.
\end{equation}
By using the observability inequality \eqref{obsL^2L^1}, \eqref{PoseControl}, \eqref{EstimCOntrol} and Young's inequality, we obtain the following bound on the sequence of controls
\begin{equation}
\label{EstimControl}
\norme{h_{\varepsilon}}_{L^{\infty}(Q_T)}^2 \leq C(\Omega,\omega,T,a) \norme{y_0}_{L^2(\Omega)}^2,
\end{equation}
where $C(\Omega,\omega,T,a)$ is of the form \eqref{coutcontrolereg}.\\
\indent For $\lambda > 0$ and $p_{T} \in L^2(\Omega;\R^{+})$, we have
\begin{equation}
J_{\varepsilon}(q_{\varepsilon,T}) \leq J_{\varepsilon}(q_{\varepsilon,T}+\lambda p_{T}) 
\label{CondMin} .
\end{equation}
Dividing the inequality \eqref{CondMin} by $\lambda$ and letting $\lambda \rightarrow 0^{+}$, we easily obtain from \eqref{PoseControl},
\begin{align}
\label{Ineul}
- (y_0, p(0,.))_{L^2(\Omega)}
& \leq \int_{(0,T)\times\omega} h_{\varepsilon} p  + \varepsilon \liminf_{\lambda \rightarrow 0^{+}} \frac{\norme{q_{\varepsilon,T} + \lambda p_T}_{L^2(\Omega)} - \norme{q_{\varepsilon,T}}_{L^2(\Omega)}}{\lambda}\\
& \leq \int_{(0,T)\times\omega} h_{\varepsilon} p  + \varepsilon \norme{p_T}_{L^2(\Omega)}, \notag
\end{align} 
where $p$ is the solution to \eqref{eqlinaAdj} with initial data $p_T$.
Since systems \eqref{eqlin} and \eqref{eqlinaAdj} are in duality, we have
\begin{equation}
\label{duality}
\int_{(0,T)\times\omega} h_{\varepsilon} p = (y_{\varepsilon}(T,.), p_T)_{L^2(\Omega)} - (y_0, p(0,.))_{L^2(\Omega)},
\end{equation}
which, combined with \eqref{Ineul}, yields
\begin{equation}
\label{inpos}
(y_{\varepsilon}(T,.), p_T)_{L^2(\Omega)} \geq - \varepsilon \norme{p_T}_{L^2(\Omega)},\ \forall p_T \in L^2(\Omega;\R^{+}).
\end{equation}
\indent \textbf{Step 2.} By using \eqref{EstimControl}, \eqref{eqlinEps}, \Cref{wpl2}, \Cref{wplinfty} and the embedding \eqref{injclassique}, up to a subsequence, we get that there exist $h \in L^{\infty}(Q_T)$ and $y \in W_T \cap L^{\infty}(Q_T)$ such that 
\begin{equation}
\label{ConvHeps}
h_{\varepsilon} {\rightharpoonup}^*\  h\ \text{in}\ L^{\infty}(Q_T)\ \text{as}\ \varepsilon \rightarrow 0,
\end{equation}
\begin{equation}
\label{ConvYeps}
y_{\varepsilon} {\rightharpoonup}\ y\ \text{in}\ W_T\ \Rightarrow y_{\varepsilon}(0,.){\rightharpoonup}\ y(0,.),\ y_{\varepsilon}(T,.){\rightharpoonup}\ y(T,.)\ \text{in}\ L^2(\Omega)\qquad \text{as}\ \varepsilon \rightarrow 0.
\end{equation}
Then, by using \eqref{eqlinEps}, \eqref{ConvHeps} and \eqref{ConvYeps}, we obtain that $y$ is the solution of \eqref{eqlin} associated to the control $h$ satisfying \eqref{EstiControlinfty} (by letting $\varepsilon$ goes to $0$ in \eqref{EstimControl}) and
\begin{equation}
\label{inposBis}
(y(T,.), p_T)_{L^2(\Omega)} \geq 0,\ \forall p_T \in L^2(\Omega;\R^{+}).
\end{equation}
Then, we deduce from \eqref{inposBis} that $y$ satisfies \eqref{yTNeqLin}, which concludes the proof of \Cref{TheoControlInfty}.
\end{proof}

\section{A fixed-point argument to prove the small-time nonlinear global nonnegative controllability}
\label{Fixedpointsection}
The goal of this section is to prove \Cref{TheoNeg}. We assume that \eqref{finfty} holds for $\alpha \leq 2$ and $f(s) \geq 0$ for $s \geq 0$.
\subsection{A comparison principle}
First, we begin with this lemma, which is a consequence of the comparison principle for subsolutions and supersolutions of \eqref{heatSL} with control $h=0$ stated in \Cref{SubSuper}.
\begin{lem}
\label{LemComp}
Let $T>0$, $y_0 \in L^{\infty}(\Omega)$. Assume that there exists $T^{*} \in (0,T]$ and a control $h^{*} \in L^{\infty}(Q_{T^{*}})$ such that the solution $y \in L^{\infty}(Q_{T^{*}})$ to \eqref{heatSL} satisfies \eqref{yTNeg} (replacing $T \leftarrow T^{*}$). Then, if we set
$$ h(t,. ) := 
\left\{
\begin{array}{ll}

h^{*}(t,.)& \text{for}\ t \in (0,T^{*}),\\ 
0& \text{for}\ t \in (T^{*},T),

\end{array}
\right.
$$
the solution $y$ of \eqref{heatSL} belongs to $L^{\infty}(Q_T)$ and satisfies \eqref{yTNeg}. Moreover, there exists $C :=C(\Omega)>0$ such that
\begin{equation}
\label{estimWholeinterval}
\norme{y}_{L^{\infty}(Q_T)} \leq C \norme{y}_{L^{\infty}(Q_{T^*})}.
\end{equation}
\end{lem}
\begin{proof}
By using the fact that $f(0)=0$, $f(s) \geq 0$ for $s \geq 0$ and the comparison principle (see \Cref{SubSuper}), we have
\begin{equation}
\label{LemcompEnfin}
\forall t \in [T^{*},T],\ \text{a.e.}\ x \in \Omega,\  0 \leq y(t,x) \leq \widetilde{y}(t,x),
\end{equation}
where $\widetilde{y}$ is the nonnegative solution to 
\begin{equation}
\left\{
\begin{array}{l l}
\partial_t \widetilde{y} - \Delta \widetilde{y} =  0 &\mathrm{in}\ (T^{*},T)\times\Omega,\\
\frac{\partial\widetilde{y}}{\partial n} = 0&\mathrm{on}\ (T^{*},T)\times\partial\Omega,\\
\widetilde{y}(T^{*},.)=y(T^{*},.)& \mathrm{in}\ \Omega.
\end{array}
\right.
\label{Heat}
\end{equation}
Therefore, by using \Cref{wplinfty} for \eqref{Heat}, we get that there exists $C:=C(\Omega)>0$ such that
\begin{equation}
\label{Estimheat}
\norme{\widetilde{y}}_{L^{\infty}((T^{*},T)\times\Omega)} \leq C \norme{y(T^{*},.)}_{L^{\infty}(\Omega)} \leq C \norme{y}_{L^{\infty}(Q_{T^*})}.
\end{equation}
By using \eqref{LemcompEnfin} and \eqref{Estimheat}, we obtain that $y \in L^{\infty}(Q_T)$, \eqref{yTNeg} and \eqref{estimWholeinterval} hold.
\end{proof}
\subsection{The fixed-point: definition of the application}
We begin with some notations.
Let us set
\begin{equation}
\label{defg}
g(s) = \left\{
\begin{array}{ll}
\displaystyle\frac{f(s)}{s} & \displaystyle \textrm{if}\ s \neq 0,\\
\displaystyle f'(0) &\displaystyle \textrm{if}\ s \= 0.
\end{array}
\right.
\end{equation}
The function $g$ is continuous and by using the fact that $f$ satisfies \eqref{finfty} with $\alpha\leq 2$, we deduce that for every $\varepsilon > 0$, there exists $C_{\varepsilon} > 0$ such that 
\begin{equation}
\label{estig}
\forall s \in \R,\ |g(s)|^{1/2} \leq \varepsilon \log(2+|s|) + C_{\varepsilon}.
\end{equation} 
The end of the section is devoted to the proof of \Cref{TheoNeg}.
\begin{proof}
Let $T>0$, $y_0 \in L^{\infty}(\Omega)$. \\
\indent Unless otherwise specified, we denote by $C$ various positive constants varying from line to line which may depend on $\Omega$, $\omega$, $T$.\\ 
\indent We will perform a Kakutani-Leray-Schauder's fixed-point argument in $L^{\infty}(Q_T)$.\\
\indent For each $z \in L^{\infty}(Q_T)$, we consider the linear system
\begin{equation}
\left\{
\begin{array}{l l}
\partial_t y  - \Delta y + g(z) y =  h 1_{\omega} &\mathrm{in}\ (0,T)\times\Omega,\\
\frac{\partial y}{\partial n} = 0&\mathrm{on}\ (0,T)\times\partial\Omega,\\
y(0,.)=y_0& \mathrm{in}\ \Omega.
\end{array}
\right.
\label{eqlinaFixPoint}
\end{equation}
We set
\begin{equation}
\label{defT*}
T_z^{*} := \min\Big(T, \norme{g(z)}_{L^{\infty}(Q_T)}^{-1/2}\Big).
\end{equation}
According to \Cref{TheoControlInfty}, there exists a control $h_z \in L^{\infty}(Q_{T_z^{*}})$ satisfying 
\begin{equation}
\label{estcontFixPointINT}
\begin{array}{ll}
&\displaystyle \norme{h_z}_{L^{\infty}(Q_{T_z^{*}})}\\
& \displaystyle \leq\exp\left(C\left(1+\frac{1}{T_z^{*}}+T_z^{*}\norme{g(z)}_{L^{\infty}(Q_T)} + \norme{g(z)}_{L^{\infty}(Q_T)}^{1/2}\right)\right)\norme{y_0}_{L^2(\Omega)} \\
& \displaystyle \leq \exp\Big(C\left(1 +\norme{g(z)}_{L^{\infty}(Q_T)}^{1/2} \right)\Big)\norme{y_0}_{L^2(\Omega)},
\end{array}
\end{equation}
\normalsize
such that the solution $y$ of \eqref{eqlinaFixPoint} in $(0,T_z^{*})\times\Omega$ with $h=h_z$ satisfies
\begin{equation}
\label{yTnegInt}
y(T_z^*, .)\geq 0.
\end{equation}
By extending by $0$ the control $h_z$ in $(T_z^{*},T)$, we get from \eqref{estcontFixPointINT}
\begin{equation}
\label{estcontFixPoint}
\norme{h_z}_{L^{\infty}(Q_T)} \leq \exp\Big(C\left(1 +\norme{g(z)}_{L^{\infty}(Q_T)}^{1/2} \right)\Big)\norme{y_0}_{L^2(\Omega)}.
\end{equation}
\indent For each $z \in L^{\infty}(Q_T)$, we introduce the set of controls
\begin{equation}
\label{SetHz}
H(z) := \{ h_z \in L^{\infty}(Q_T)\ ;\ h_z\ \text{fulfills}\ \eqref{estcontFixPoint}\ \text{and}\ h_z \equiv 0\ \text{in}\ (T_z^{*},T)\times\Omega\}.
\end{equation}
We have the following facts.
\begin{claim}
\label{Prop1Hz}
For every $z \in L^{\infty}(Q_T)$, $H(z)$ is compact for the weak-star topology of $L^{\infty}(Q_T)$.
\end{claim}
\begin{claim}
\label{Prop2Hz}
Assume that $z_k \rightarrow z$ in $L^{\infty}(Q_T)$ and $h_k \in H(z_k) {\rightharpoonup}^*\  h\ \text{in}\ L^{\infty}(Q_T)\ \text{as}\ k \rightarrow + \infty$. Then, we have $h \in H(z)$.
\end{claim}
\indent We define the set-valued mapping $\Phi: L^{\infty}(Q_T)\rightarrow \mathcal{P}(L^{\infty}(Q_T))$ in the following way. For every $z \in L^{\infty}(Q_T)$, $\Phi(z)$ is the set of $y \in L^{\infty}(Q_T)$ such that for some $h_z \in H(z)$, $y$ is the solution of \eqref{eqlinaFixPoint} and this solution satisfies \eqref{yTnegInt}.\\
\indent We recall the Kakutani-Leray-Schauder's fixed point theorem (see \cite[Theorem 2.2, Theorem 2.4]{Granas}).
\begin{theo}[Kakutani-Leray-Schauder's fixed point theorem]\label{kakLeray}
If
\begin{enumerate}[nosep]
\item $\Phi$ is a Kakutani map, that is to say for every $z \in L^{\infty}(Q_T)$, $\Phi(z)$ is a nonempty convex and closed subset of $L^{\infty}(Q_T)$,
\item $\Phi$ is compact, that is to say for every bounded set $B \subset L^{\infty}(Q_T)$, there exists a compact set $K \subset L^{\infty}(Q_T)$ such that for every $z \in B$, $\Phi(z) \subset K$,
\item $\Phi$ is upper semicontinuous in $L^{\infty}(Q_T)$, that is to say for all closed subset $\mathcal{A} \subset L^{\infty}(Q_T)$, $\Phi^{-1}(\mathcal{A}) = \{z \in L^{\infty}(Q_T)\ ;\ \Phi(z)\cap \mathcal{A} \neq \emptyset \}$ is closed,
\item $\mathcal{F} := \{ y \in L^{\infty}(Q_T)\ ;\ \exists \lambda \in (0,1),\ y \in \lambda \Phi(y)\}$ is bounded in $L^{\infty}(Q_T)$,
\end{enumerate}
hold.\\
\indent Then $\Phi$ has a fixed point, i.e, there exists $y \in L^{\infty}(Q_T)$ such that $y \in \Phi(y)$.
\end{theo}
\subsection{Hypotheses of Kakutani-Leray-Schauder's fixed point theorem}
We will check that the four hypotheses of \Cref{kakLeray} hold.\\
\indent \textbf{The point $(1)$ holds.} Indeed, for every $z \in L^{\infty}(Q_T)$, we have seen that $\Phi(z)$ is nonempty. The convexity of $\Phi(z)$ comes from the fact that the inequality \eqref{yTnegInt} is stable by convex combinations. Let us show that $\Phi(z)$ is closed. Let $(y_k)_{k \in \N}$ be a sequence of elements in $L^{\infty}(Q_T)$, such that for every $ k \in \N$, $y_k \in \Phi(z)$ and $y_k \rightarrow y$ in $L^{\infty}(Q_T)$. Then, for every $k \in \N$, there exists a control $h_k \in H(z)$ such that $y_k$ is the solution to 
\begin{equation}
\left\{
\begin{array}{l l}
\partial_t y_k - \Delta y_k + g(z) y_k=  h_k 1_{\omega} &\mathrm{in}\ (0,T)\times\Omega,\\
\frac{\partial y_k}{\partial n}= 0&\mathrm{on}\ (0,T)\times\partial\Omega,\\
y_k(0,.)=y_0& \mathrm{in}\ \Omega,
\end{array}
\right.
\label{eqlinaFixPointCloses}
\end{equation}
and this solution satisfies 
\begin{equation}
\label{FirstPt3FCloses}
y_k(T_z^*,.)\geq 0.
\end{equation}
By using \Cref{Prop1Hz}, \Cref{wpl2} and the embedding \eqref{injclassique}, we get that there exist a strictly increasing sequence $(k_l)_{l \in \N}$ of integers and $h \in H(z)$ such that 
\begin{equation}
\label{FirstPt31Closes}
h_{k_l} {\rightharpoonup}^*\  h\ \text{in}\ L^{\infty}(Q_T)\ \text{as}\ l \rightarrow + \infty,
\end{equation}
\begin{equation}
\label{FirstPt32Closes}
y_{k_l} \rightharpoonup y\ \text{in}\ W_T\  \Rightarrow y_{k_l}(0,.)\rightharpoonup y_0,\ y_{k_l}(T_z^*,.){\rightharpoonup}\ y(T_z^*,.)\ \text{in}\ L^2(\Omega)\quad \text{as}\ l \rightarrow +\infty.
\end{equation}
By passing to the limit as $l \rightarrow + \infty$ in \eqref{eqlinaFixPointCloses}, \eqref{FirstPt3FCloses} and by using \eqref{FirstPt31Closes} and \eqref{FirstPt32Closes}, we get that $y \in \Phi(z)$. This concludes the proof of the point (1).\\
\indent \textbf{The point $(2)$ holds.} Let $B$ be a bounded set of $L^{\infty}(Q_T)$. By using \eqref{estcontFixPoint} and \Cref{wplinfty} applied to \eqref{eqlinaFixPoint}, we deduce that there exists $R>0$ such that for every $z \in B$, for every $y \in \Phi(z)$ associated to a control $h_z \in H(z)$, we have
\begin{equation}
\label{Estimzyh}
z, y, h_z \in B_R := \{ \zeta \in L^{\infty}(Q_T)\ ;\ \norme{\zeta}_{L^{\infty}(Q_T)} \leq R\}.
\end{equation}
Let $Y\in L^{\infty}(Q_T)$ be the solution to the Cauchy problem
\begin{equation}
\left\{
\begin{array}{l l}
\partial_t {Y} -  \Delta Y = 0 &\mathrm{in}\ (0,T)\times\Omega,\\
\frac{\partial Y}{\partial n} = 0 &\mathrm{on}\ (0,T)\times\partial\Omega,\\
Y(0,.)=y_0 &\mathrm{in}\  \Omega.
\end{array}
\right.
\end{equation}
Let $y^{*} = y-Y$, where $y \in \Phi(z)$, with $z \in B$, associated to a control $h_z \in H(z)$. Then, $y^{*}$ is the solution to
\begin{equation}
\left\{
\begin{array}{l l}
\partial_t {y^{*}} - \Delta y^{*} + g(z)y =h_z  1_{\omega} &\mathrm{in}\ (0,T)\times\Omega,\\
\frac{\partial y^{*}}{\partial n} = 0 &\mathrm{on}\ (0,T)\times\partial\Omega,\\
y^{*}(0,.)=0 &\mathrm{in}\  \Omega.
\end{array}
\right.
\label{zetaetoile}
\end{equation}
From \eqref{Estimzyh}, we have
\begin{equation}
\norme{-g(z)y + h_z 1_{\omega}}_{L^{\infty}(Q_T)} \leq C_R.
\label{estimationsourcezeta*}
\end{equation}
From \eqref{estimationsourcezeta*}, a maximal parabolic regularity theorem in $L^p$ (see \cite[Theorem 2.1]{DHP}), with $p=N+2$, applied to $y^{*}$, solution of \eqref{zetaetoile}, we deduce that
\begin{equation}
\label{EstimateLpystar}
y^{*}\in X_p := W^{1,p}(0,T;L^p(\Omega)) \cap L^p(0,T;W^{2,p}(\Omega))\ \text{and}\ \norme{y^{*}}_{X_p} \leq C_R.
\end{equation}
By the Sobolev embedding theorem $X_p \hookrightarrow C^{\beta/2,\beta}(\overline{Q_T})$ with $\beta>0$ (see \cite[Theorem 1.4.1]{WYW}), we deduce that $y^{*} \in C^{0}(\overline{Q_T})$ and
\begin{equation}
\forall (t,x) \in \overline{Q_T},\ \forall (t',x') \in \overline{Q_T},\ |y^{*}(t,x) - y^{*}(t',x')|\leq C_R(|t-t'|^{\beta/2} + |x-x'|^{\beta}).
\label{holder}
\end{equation}
Let $K^{*}$ be the set of $y^{*}$ such that \eqref{holder} holds. Then, we have $K:=(Y+ K^{*}) \cap B_R$ is a compact convex subset of $L^{\infty}(Q_T)$ by Ascoli's theorem and
\[ \forall z \in B,\ \Phi(z) \subset K.\]
This concludes the proof of the point (2).\\
\indent \textbf{The point (3) holds.} Let $\mathcal{A}$ be a closed subset of $L^{\infty}(Q_T)$. Let $(z_k)_{k \in \N}$ be a sequence of elements in $L^{\infty}(Q_T)$, $(y_k)_{k \in \N}$ be a sequence of elements in $L^{\infty}(Q_T)$, and $z \in L^{\infty}(Q_T)$ be such that
\begin{equation}
\label{FirstPt3}
z_k \rightarrow z \ \mathrm{in}\ L^{\infty}(Q_T)\ \text{as}\ k \rightarrow + \infty,
\end{equation}
\begin{equation}
\label{FirstPt3B}
\forall k \in \N,\ y_k \in \mathcal{A},
\end{equation}
\begin{equation}
\label{FirstPt3T}
\forall k \in \N,\ y_k \in \Phi(z_k).
\end{equation}
By \eqref{FirstPt3T} and \eqref{estcontFixPoint}, for every $k \in \N$, there exists a control $h_k \in H(z_k)$ such that $y_k$ is the solution to 
\begin{equation}
\left\{
\begin{array}{l l}
\partial_t y_k - \Delta y_k + g(z_k) y_k=  h_k 1_{\omega} &\mathrm{in}\ (0,T)\times\Omega,\\
\frac{\partial y_k}{\partial n}= 0&\mathrm{on}\ (0,T)\times\partial\Omega,\\
y_k(0,.)=y_0& \mathrm{in}\ \Omega,
\end{array}
\right.
\label{eqlinaFixPointBis}
\end{equation}
and this solution satisfies 
\begin{equation}
\label{FirstPt3F}
y_k(T_{z_k}^*,.)\geq 0.
\end{equation}
By \eqref{FirstPt3}, \Cref{Prop2Hz} and the point (2) of \Cref{kakLeray}, we get that there exist a strictly increasing sequence $(k_l)_{l \in \N}$ of integers, $h \in H(z)$ and $y \in L^{\infty}(Q_T)$ such that 
\begin{equation}
\label{FirstPt31}
h_{k_l} {\rightharpoonup}^*\  h\ \text{in}\ L^{\infty}(Q_T)\ \text{as}\ l \rightarrow + \infty,
\end{equation}
\begin{equation}
\label{FirstPt32}
y_{k_l} \rightarrow y\ \text{in}\ L^{\infty}(Q_T)\ \text{as}\ l \rightarrow + \infty.
\end{equation}
Since $\mathcal{A}$ is closed, \eqref{FirstPt3B} and \eqref{FirstPt32} imply that $y \in \mathcal{A}$. Hence, it suffices to check that
\begin{equation}
\label{FirstPt333}
y \in \Phi(z).
\end{equation}
Letting $l \rightarrow + \infty$ in \eqref{eqlinaFixPointBis} and \eqref{FirstPt3F} and using \eqref{FirstPt3}, \eqref{FirstPt31} and \eqref{FirstPt32}, we get that $y$ satisfies \eqref{eqlinaFixPoint} and \eqref{yTnegInt}. Hence, \eqref{FirstPt333} holds. This concludes the proof of the point (3).\\
\indent \textbf{The point (4) holds.} Let $y \in \mathcal{F}$. Then, for some $\lambda \in (0,1)$ and $h_y \in H(y)$, we have 
\begin{equation*}
\label{heatSLFixedPoint}
\left\{
\begin{array}{l l}
\partial_t y-  \Delta y + f(y) =  \lambda h_y 1_{\omega} &\mathrm{in}\ (0,T)\times\Omega,\\
\frac{\partial y}{\partial n}= 0&\mathrm{on}\ (0,T)\times\partial\Omega,\\
y(0,.)=\lambda y_0& \mathrm{in}\ \Omega.
\end{array}
\right.
\end{equation*}
and 
\begin{equation*}
y(T_y^{*},.) \geq 0.
\end{equation*}
Therefore, by using \Cref{LemComp} and \Cref{wplinfty}, we have
\begin{equation}
\label{estiSolFixPoint1}
\begin{array}{ll}
\displaystyle \norme{y}_{L^{\infty}(Q_T)}& \leq \displaystyle C \norme{y}_{L^{\infty}(Q_{T_y^{*}})} \\
&\displaystyle \leq C \exp\left(CT_y^* \norme{g(y)}_{L^{\infty}(Q_T)}\right)\left(\norme{y_0}_{L^{\infty}(\Omega)} + \norme{h_y}_{L^{\infty}(Q_T)}\right).
\end{array}
\end{equation} 
Consequently, by taking into account the definition of $T_y^{*}$, i.e., \eqref{defT*} and using \eqref{estcontFixPoint}, \eqref{estiSolFixPoint1}, \eqref{estig}, we deduce that 
\begin{equation}
\label{estiSolFixPoint2}
\begin{array}{ll}
\displaystyle\norme{y}_{L^{\infty}(Q_T)}  &\displaystyle\leq \exp\left(C\left(1 +\norme{g(y)}_{L^{\infty}(Q_T)}^{1/2} \right)\right) \norme{y_0}_{L^{\infty}(\Omega)}\\
&\displaystyle \leq \exp\left(C\left(1+\varepsilon \log\left(2+\norme{y}_{L^{\infty}(Q_T)} \right) + C_{\varepsilon}\right)\right)\norme{y_0}_{L^{\infty}(\Omega)}\\
&\displaystyle \leq  \exp\left(C_\varepsilon\right) \left(2+\norme{y}_{L^{\infty}(Q_T)} \right)^{ \varepsilon C} \norme{y_0}_{L^{\infty}(\Omega)}.
\end{array}
\end{equation}
Therefore, by taking $\varepsilon$ sufficiently small such that $\varepsilon C = 1/2$, we deduce from \eqref{estiSolFixPoint2} that $\mathcal{F}$ is bounded in $L^{\infty}(Q_T)$. This concludes the proof of the point (4).\\
\indent By \Cref{kakLeray}, $\Phi$ has a fixed point $y$. We denote by $h_y$ the associated control. Then, by using \Cref{LemComp}, $y$ is the solution to \eqref{heatSL} with control $h_y$ such that \eqref{yTNeg} holds. This concludes the proof of \Cref{TheoNeg}.
\end{proof}

\section{Application of the global nonnegative-controllability to the large time global null-controllability}
\label{sectionCorGlobal}
In this section, we prove \Cref{CorGlobal}. We assume that \eqref{finfty} holds for $\alpha \in [3/2,2]$, $f(s) > 0$ for $s > 0$ and $1/f \in L^1([1,+\infty))$. 
\begin{proof} Let $y_0 \in L^{\infty}(\Omega)$. The proof is divided into three steps.\\
\indent \textbf{Step 1: Steer the solution to a nonnegative state in time $T_1:=1$.} By using \Cref{TheoNeg}, there exists $h_1 \in L^{\infty}(Q_{T_{1}})$ such that the solution $y$ to \eqref{heatSL} replacing $T \leftarrow T_1$ satisfies
$$ y_{T_1} := y(T_1,.) \geq 0.$$
\indent \textbf{Step 2: Dissipation of $f$ on $\R^{+}$ and comparison to an ordinary differential equation.} We set 
$$ h_2(t,.) := 0,\ \text{for}\ t \in [T_1,T_2],$$
with $T_2$ which will be determined later.\\
\indent Then, by using the comparison principle given in \Cref{SubSuper}, we deduce that the solution $y$ to 
\begin{equation*}
\left\{
\begin{array}{l l}
\partial_t y-  \Delta y  =  - f(y)&\mathrm{in}\ (T_1,T_2)\times\Omega,\\
\frac{\partial y}{\partial n}= 0&\mathrm{on}\ (T_1,T_2)\times\partial\Omega,\\
y(T_1,.)=y_{T_1}& \mathrm{in}\ \Omega,
\end{array}
\right.
\end{equation*}
satisfies
\begin{equation}
\label{CompareEDO}
\forall t \in  [T_1,T_2], \ \text{a.e.}\ x \in \Omega,\ 0 \leq y(t,x) \leq v(t),
\end{equation}
where $v$ is the (global) nonnegative solution to the ordinary differential equation
\begin{equation}
\label{EDO}
\left\{
\begin{array}{l l}
\dot{v}(t)  =  - f(v(t))&\mathrm{in}\ (T_1,+\infty),\\
v(T_1)=\norme{y_{T_1}}_{L^{\infty}(\Omega)}+1&.
\end{array}
\right.
\end{equation}
A straightforward calculation leads to 
\begin{equation}
\label{FormuleV}
\forall t \in [T_1,+\infty),\ v(t) > 0\ \text{and}\  F(v(t))-F(v(T_1)) = t-T_1,
\end{equation}
where $F$ is defined as follows
\begin{equation}
\label{DefiF}
\forall s > 0,\ F(s) = \int_{+\infty}^{s} \frac{-1}{f(\sigma)} d \sigma = \int_{s}^{+\infty} \frac{1}{f(\sigma)} d \sigma.
\end{equation}
Note that $F$ is well-defined because $f(\sigma) > 0$ for every $\sigma > 0$ and $1/f \in L^{1}([1,+\infty))$ by hypothesis. We check that $F$ is a $C^{1}$ strictly decreasing function. Moreover, we have $1/f \notin L^1((0,1])$ because $f \in C^1(\R;\R)$ and $f(0)=0$. Hence, we have by \eqref{DefiF}
\begin{equation}
\label{AsymptoticF}
 \lim_{s \rightarrow 0^{+}} F(s) = +\infty\ \text{and}\ \lim_{s \rightarrow +\infty} F(s) = 0.
 \end{equation}
Therefore, we deduce that $F : (0,+\infty) \rightarrow (0,+\infty) $ is a $C^1$-diffeomorphism. We denote by $F^{-1}:(0,+\infty) \rightarrow (0,+\infty) $ its inverse, which is strictly decreasing. Then, by \eqref{FormuleV}, we have 
\begin{equation}
\label{FormuleVBis}
\forall t \in [T_1,+\infty),\ v(t) = F^{-1}(t-T_1 + F(v(T_1)) \leq F^{-1}(t-T_1).
\end{equation}
The estimate \eqref{FormuleVBis} is the key point because it states that we can upperbound $v$ by a function independent of the size of $v(T_1)$ and we also have 
\begin{equation}
\label{FormuleVTer}
F^{-1}(t-T_1) \rightarrow 0\ \text{as}\ t \rightarrow + \infty,
\end{equation}
by using \eqref{AsymptoticF}.\\
\indent Let $\delta > 0$ be such that the null-controllability of \eqref{heatSL} holds in $B_{L^{\infty}(\Omega)}(0,\delta)$ in time $T=1$. The existence of $\delta$ is given by \Cref{ResLocal}. \\
\indent By \eqref{FormuleVTer}, we deduce that there exists $T_2$ sufficiently large such that
\begin{equation}
\label{FormuleVQ}
F^{-1}(T_2-T_1) \leq \delta.
\end{equation}
Consequently, by using \eqref{CompareEDO}, \eqref{FormuleVBis}, \eqref{FormuleVQ}, we have
\begin{equation}
\label{CompareEDOBis}
\text{a.e.}\ x \in \Omega,\ 0 \leq y(T_2,x) \leq \delta.
\end{equation}
\indent \textbf{Step 3: Local null-controllability.} By using \Cref{ResLocal} with $T=1$, we deduce from \eqref{CompareEDOBis} that there exists a control $h_3 \in L^{\infty}((T_2,T_3)\times\Omega)$ with $T_3 :=T_2+1$ such that the solution $y$ of \eqref{heatSL} replacing $(0,T) \leftarrow (T_2,T_3)$ satisfies $y(T_3,.)=0$.\\
\indent To sum up, the control
$$ h(t,. ) := 
\left\{
\begin{array}{ll}

h_1(t,.)& \text{for}\ t \in (0,T_1),\\ 
h_2(t,.)& \text{for}\ t \in (T_1,T_2),\\
h_3(t,.)& \text{for}\ t \in (T_2,T_3),

\end{array}
\right.
$$
steers the initial data $y_0 \in L^{\infty}(\Omega)$ to $0$. It is worth mentioning that the final time of control $T_3$ does not depend on $y_0$. This concludes the proof of \Cref{CorGlobal}.
\end{proof}

\section{Dirichlet boundary conditions}
\label{SubDir}
\Cref{TheoNeg} and \Cref{CorGlobal} remain valid for Dirichlet boundary conditions, as to say for 
\begin{equation}
\label{heatSLDir}
\left\{
\begin{array}{l l}
\partial_t y-  \Delta y + f(y) =  h 1_{\omega} &\mathrm{in}\ (0,T)\times\Omega,\\
y= 0&\mathrm{on}\ (0,T)\times\partial\Omega,\\
y(0,.)=y_0& \mathrm{in}\ \Omega.
\end{array}
\right.
\end{equation}
The main point is to establish a $L^1$-Carleman estimate similar to \Cref{CarlL1} for
\begin{equation}
\left\{
\begin{array}{l l}
-\partial_t q - \Delta q + a(t,x) q = 0  &\mathrm{in}\ (0,T)\times\Omega,\\
q= 0 &\mathrm{on}\ (0,T)\times\partial\Omega,\\
q(T,.)=q_T& \mathrm{in}\ \Omega.
\end{array}
\right.
\label{eqlinaAdjDir}
\end{equation}
We keep the notations of \Cref{SectionCarlL1}.
\begin{theo}
\label{CarlL1Dir}
There exists two constants $C=C(\Omega,\omega)>0$ and $C_1 := C_1(\Omega,\omega)>0$, such that, \begin{equation}
\forall \lambda \geq 1,\qquad \forall s \geq s_1(\lambda):=C(\Omega,\omega)\left(e^{2\lambda \norme{\eta^0}_{\infty}}T+T^{2}+T^2 \norme{a}_{L^{\infty}(Q_T)}^{1/2}\right),
\label{deflambd1s1Dir}
\end{equation}
for every $q_T \in L^2(\Omega;\R^{+})$, the nonnegative solution $q$ of \eqref{eqlinaAdjDir} satisfies
\begin{align}
\label{carlDir}
\lambda \int_{Q_T} e^{-s\alpha}  s \xi^{2}  \eta^{0} q + \int_{Q_T}  e^{-s\alpha} \xi q \leq C_1 \lambda \int_{(0,T)\times\omega} e^{-s\alpha} s \xi^{2}  q  dxdt.
\end{align}
\end{theo}
\begin{proof}
\indent The proof follows the one of \Cref{CarlL1}. This is why we omit some details. We multiply the identity \eqref{eqM} by $\eta^{0}$ and we integrate over $(0,T)\times \Omega$
\small
\begin{equation}
\label{Integrateeta}
\begin{array}{ll}
&\displaystyle\int_{Q_T} s^{2} \lambda^2 |\nabla \eta^0|^2 \xi^2 \psi \eta^{0}- \int_{Q_T} 2s \lambda \xi (\nabla \eta^0 . \nabla \psi)  \eta^{0}+\int_{Q_T}  (\partial_t \psi) \eta^{0}+\int_{Q_T} (\Delta \psi) \eta^{0}\\
&= \displaystyle \int_{Q_T} s \lambda^2  |\nabla \eta^0|^2 \xi \psi \eta^{0}-   \int_{Q_T} s \alpha_t \psi\eta^{0}+  \int_{Q_T}a(t,x) \psi \eta^{0}\\
& \displaystyle \qquad +\int_{Q_T} s \lambda \Delta \eta^0 \xi \psi \eta^{0}.
\end{array}
\end{equation}
\normalsize
By the properties of $\eta^{0}$, we have
\begin{equation}
\label{SepInteta}
\int_{Q_T} s^{2} \lambda^2 |\nabla \eta^0|^2 \xi^2 \psi \eta^{0} \geq m \int_{Q_T} s^{2} \lambda^2 \xi^2 \psi \eta^{0} - m \int_{(0,T)\times\omega} s^{2} \lambda^2  \xi^2 \psi\eta^{0},
\end{equation}
where $m$ is defined in \eqref{defminm}.\\
\indent By combining \eqref{Integrateeta} and \eqref{SepInteta}, we have 
\small
\begin{equation}
\label{IntegrateBiseta}
\begin{array}{ll}
&\displaystyle m \int_{Q_T} s^{2} \lambda^2  \xi^2 \psi \eta^{0}- \int_{Q_T} 2s \lambda \xi (\nabla \eta^0 . \nabla \psi)  \eta^{0}+\int_{Q_T}  (\partial_t \psi) \eta^{0}+\int_{Q_T} (\Delta \psi) \eta^{0}\\
&\leq \displaystyle \int_{Q_T} s \lambda^2  |\nabla \eta^0|^2 \xi \psi \eta^{0} + \int_{Q_T} s |\alpha_t| \psi \eta^{0}+  \int_{Q_T}|a(t,x)| \psi \eta^{0}\\
& \displaystyle \qquad +\int_{Q_T} s \lambda |\Delta \eta^0|\xi \psi \eta^{0}+ m\int_{(0,T)\times\omega} s^{2} \lambda^2 \xi^2 \psi\eta^{0}.
\end{array}
\end{equation}
\normalsize
We have the following integration by parts
\small
\begin{align}
\label{IntByParts1eta}
\begin{array}{ll}
\displaystyle-\int_{Q_T} 2s \lambda \xi (\nabla \eta^0 . \nabla \psi)\eta^{0} =   \int_{Q_T} 2s \lambda \left((\nabla \xi . \nabla \eta^{0}) \eta^{0} \psi  + \xi( \Delta \eta^{0} )\eta^{0}\psi   +  \underbrace{\xi |\nabla \eta^0|^{2} \psi}_{\geq 0}\right).
\end{array}
\end{align}
\begin{equation}
\int_{Q_T}  (\partial_t \psi)\eta^{0}  =  \int_{\Omega} \eta^{0} (.)(\psi(T,.) - \psi(0,.) ) = 0,
\label{IntByParts2eta}
\end{equation}
\begin{equation}
\int_{Q_T} (\Delta \psi)\eta^{0} = \int_{Q_T} \psi \Delta\eta^{0}.
 \label{IntByParts3eta}
\end{equation}
\normalsize
From \eqref{IntegrateBiseta}, \eqref{IntByParts1eta}, \eqref{IntByParts2eta}, \eqref{IntByParts3eta} and the properties of $\eta^{0}$, we have
\small
\begin{equation}
\label{IntegrateTrieta}
\begin{array}{ll}
&\displaystyle m \int_{Q_T} s^{2} \lambda^2 \xi^2 \psi \eta^{0}  + 2m\int_{Q_T} s \lambda \xi \psi\\
&\leq \displaystyle \int_{Q_T} s \lambda^2  |\nabla \eta^0|^2 \xi \psi \eta^{0} + \int_{Q_T} s |\alpha_t| \psi \eta^{0}+  \int_{Q_T}|a(t,x)| \psi  \eta^{0}\\
& \displaystyle \qquad + 3\int_{Q_T} s \lambda |\Delta \eta^0|\xi \psi \eta^{0}+ 2\int_{Q_T} s \lambda|\nabla \xi| | \nabla \eta^{0}| \psi \eta^{0}+ \int_{Q_T} \psi |\Delta \eta^{0}|\\
&\displaystyle\qquad+m \int_{(0,T)\times\omega} s^{2} \lambda^2 \xi^2 \psi \eta^{0}+2m \int_{(0,T)\times\omega} s \lambda \xi \psi.
\end{array}
\end{equation}
\normalsize
The first five right hand side terms of \eqref{IntegrateTrieta} can be absorbed by the first left hand side term provided $s \geq s_1 (\lambda)$ as defined in \eqref{deflambd1s1Dir} (see ‘Step 2, Absorption’ of the proof of \Cref{CarlL1} for details: it is exactly the same mechanism as in the proof for the Neumann case). The sixth right hand side term of \eqref{IntegrateTrieta} can be absorbed by the second left hand side term provided $s \geq C(\Omega,\omega) T^{2}$. The two last right hand side terms of \eqref{IntegrateTrieta} are smaller than $\int_{(0,T)\times\omega} s^{2} \lambda^2 \xi^2 \psi$ provided $s \geq C(\Omega,\omega) T^{2}$. This leads to
$$ \int_{Q_T} s^{2} \lambda^2 \xi^2 \psi \eta^{0}  + \int_{Q_T} s \lambda \xi \psi \leq C \int_{(0,T)\times\omega} s^{2} \lambda^2 \xi^2 \psi,$$
which yields \eqref{carlDir} by dividing by $s \lambda$.
\end{proof}
From \Cref{CarlL1Dir}, we deduce a precise $L^2$-$L^1$ observability inequality as in \Cref{TheoObsL^2L^1} by using the second left hand side term of \eqref{carlDir}. It is an easy adaptation of \Cref{SectionL^2L^1}. \\
\indent The proof of the linear global nonnegative-controllability result as \Cref{TheoControlInfty} and the fixed-point argument (see \Cref{Fixedpointsection})  remain unchanged. This leads to the small-time global nonnegative controllability for \eqref{heatSLDir}.\\
\indent The proof of the large time global null-controllability result for \eqref{heatSLDir} follows the same lines as \Cref{sectionCorGlobal}. In particular, the comparison principle between the free solution and the solution to the ordinary differential equation, i.e., \eqref{CompareEDO} stays valid because $v(t) > 0$ on $(T_1,T_2)\times\partial\Omega$.

\section{Comments}

\subsection{Nonlinearities depending on the gradient of the state}

We do not treat semilinearties $F(y, \nabla y)$ as considered in \cite{FCGBGP2} (see also \cite{DoFCGBZ}) because the left hand side of the $L^1$-Carleman estimate \eqref{carl} established in \Cref{CarlL1} does not provide estimates on the gradient of the state. 

\subsection{Nonlinear reaction-diffusion systems}
We may wonder to what extent our main results, i.e., \Cref{TheoNeg} and \Cref{CorGlobal} for \eqref{heatSL}, can be adapted to the  $m \times m$ semilinear reaction-diffusion system
\begin{equation}
\label{Semilinear}
\forall 1 \leq i \leq m,\ \left\{
\begin{array}{l l}
\partial_t u_i - d_i \Delta u_{i} =f_i(u_1, \dots, u_m) + h_i 1_{\omega}  &\mathrm{in}\ (0,T)\times\Omega,\\
\frac{\partial u_i}{\partial n} = 0 &\mathrm{on}\ (0,T)\times\partial\Omega,\\
u_i(0,.)=u_{i,0} &\mathrm{in}\  \Omega,
\end{array}
\right.
\end{equation}
with $(d_1, \dots, d_m) \in (0,+\infty)^m$ and $(f_1, \dots, f_m) \in C^{1}(\R^{m};\R)^{m}$ satisfying
\begin{equation}
\label{f1f2Nul}
\forall i \in \{1, \dots,m\},\ f_i(0, \dots, 0) =0.
\end{equation}
\indent We assume that the nonlinearity is \textit{strongly quasi-positive}, i.e.,
\begin{equation}
\label{QuasiPos}
\forall u \in \R^{m},\ \forall i \neq j \in \{1, \dots,m\},\ \frac{\partial f_i}{\partial u_j}(u_1, \dots,u_m) \geq 0.
\end{equation}
and satisfies a ‘\textit{mass-control structure}’
\begin{equation}
\label{MassStructure}
\forall u \in [0,+\infty)^{m},\ \sum\limits_{i=1}^{m} f_i(u) \leq C \left(1+ \sum_{i=1}^{m} u_i\right).
\end{equation}
Lots of systems come naturally with the two properties \eqref{QuasiPos} and \eqref{MassStructure} in applications (see \cite[Section 2]{P}).\\
\indent We have the following global-nonnegative controllability result in small time.
\begin{theo}
\label{TheoNegSyst}
For each $f_i$, we assume that \eqref{finfty} holds for $\alpha \leq 2$. For every $T>0$, the system \eqref{Semilinear} is \textit{globally nonnegative-controllable} in time $T$.
\end{theo}
\begin{app}
Let $\alpha \in (0,2)$. The system
\begin{equation}
\left\{
\begin{array}{l l}
\partial_t u -  \Delta u =-u\log^{\alpha}(2+|u|)+ h_1 1_{\omega}  &\mathrm{in}\ (0,T)\times\Omega,\\
\partial_t v -  \Delta v =u\log^{\alpha}(2+|u|)+ h_2 1_{\omega}  &\mathrm{in}\ (0,T)\times\Omega,\\
\frac{\partial u}{\partial n} = \frac{\partial v}{\partial n} = 0 &\mathrm{on}\ (0,T)\times\partial\Omega,\\
(u,v)(0,.)=(u_0,v_0) &\mathrm{in}\  \Omega,
\end{array}
\right.
\label{Syst2x2}
\end{equation}
is globally nonnegative-controllable for every time $T>0$.
\end{app}
\begin{proof}
As the proof is very similar to that of \Cref{TheoNeg}, we limit ourselves to pointing out only the differences.\\
\indent \textbf{Difference 1: A $L^1$-Carleman estimate for a linear parabolic system.} Let $A \in L^{\infty}(Q_T; \R^{m\times m})$ be such that 
\begin{equation}
\label{QuasiPosMatrix}
\forall i \neq j \in \{1, \dots, m\},\ \text{a.e.}\ (t,x) \in Q_T,\ A_{i,j}(t,x) \geq 0.
\end{equation}
\begin{rmk}
The condition \eqref{QuasiPosMatrix} is satisfied by the linearized system of \eqref{Semilinear} around $(0,0)$ thanks to \eqref{QuasiPos}.
\end{rmk}
We consider the adjoint system
\begin{equation}
\left\{
\begin{array}{l l}
-\partial_t  \zeta - \Delta \zeta =A(t,x)\zeta  &\mathrm{in}\ (0,T)\times\Omega,\\
\frac{\partial \zeta}{\partial n}= 0 &\mathrm{on}\ (0,T)\times\partial\Omega,\\
\zeta(T,.)=\zeta_T& \mathrm{in}\ \Omega.
\end{array}
\right.
\label{eqlinaAdjSyst}
\end{equation}
Our goal is to establish this $L^1$-Carleman inequality: for every $\zeta_T \in L^2(\Omega;\R^{+})^m$, the nonnegative solution $\zeta$ of \eqref{eqlinaAdjSyst} satisfies
\begin{align}
\label{carlSyst}
\sum\limits_{i=1}^{m} \int_{Q_T} e^{-s\alpha}\xi^2 \zeta_i dx dt\leq C(\Omega,\omega)\left(\sum\limits_{i=1}^{m}  \int_{(0,T)\times\omega} e^{-s\alpha} \xi^2 \zeta_i  dxdt\right),
\end{align}
for any $\lambda \geq 1$, $s \geq s_1(\lambda):=C(\Omega,\omega)e^{4\lambda \norme{\eta^0}_{\infty}}\left(T+T^{2}+T^2 \norme{A}_{L^{\infty}(Q_T; \R^{m\times m})}^{1/2}\right)$.\\
\indent In order to prove \eqref{carlSyst}, we first remark that the nonnegativity of $\zeta$ comes from \eqref{QuasiPosMatrix} (see \cite[Chapter 3, Theorem 13]{ProtterWein}). Then, by applying the same proof strategy to each line of \eqref{eqlinaAdjSyst} as performed in \Cref{CarlL1} and by forgetting for the moment the terms involving $A_{i,j}(t,x)\zeta_j$, we get
\begin{align}
\label{carlSystInt}
\sum\limits_{i=1}^{m} \int_{Q_T} e^{-s\alpha}\lambda^{2} (s \xi)^2 \zeta_i dx dt&\leq C(\Omega,\omega)\Bigg(\norme{A}_{L^{\infty}(Q_T)}  \int_{Q_T} e^{-s\alpha} |\zeta| dx dt\\
&\qquad+\sum\limits_{i=1}^{m}  \int_{(0,T)\times\omega} e^{-s\alpha} \lambda^{2}(s \xi)^2\zeta_i  dxdt\Bigg),\notag
\end{align}
for $\lambda \geq 1$, $s \geq C(\Omega,\omega)e^{4\lambda \norme{\eta^0}_{\infty}}\left(T+T^{2}\right)$. We conclude the proof of \eqref{carlSyst} by absorbing the first right hand side term of \eqref{carlSystInt} provided $s \geq C(\Omega,\omega) T^{2} \norme{A}_{L^{\infty}(Q_T)}^{1/2}$.\\
\indent \textbf{Difference 2: Without control, the free solution associated to a nonnegative initial data of \eqref{Semilinear} stays nonnegative and remains bounded.} An adaptation of \Cref{LemComp} to the system \eqref{Semilinear} holds true. But, the reason is different. It comes from \cite[Theorem 1.1]{FeTa} which ensures global existence of classical solutions associated to nonnegative initial data for nonlinear reaction-diffusion systems with semilinearities satisfying \eqref{QuasiPos}, \eqref{MassStructure} and a (super)-quadratic growth (see also \cite{SoP} under an additional structure assumption, the so-called dissipation of entropy).
\begin{rmk}
It is worth mentioning that if the nonlinearities of \eqref{Semilinear} are bounded in $L^1(Q_T)$ for all $T>0$ (which is the case of \eqref{Syst2x2} for instance), then the solutions exist globally because the growth of the semilinearity $(f_i)_{1 \leq i \leq m}$ is less than $|u|^{\frac{N+2}{N}}$ (see \cite[Section 1]{P}).
\end{rmk}
This concludes the proof of \Cref{TheoNegSyst}.
\end{proof}
\indent In the following result, we give a sufficient condition to ensure the global null-controllability of \eqref{Semilinear}.
\begin{theo}
\label{CorGlobalSyst}
Let $\alpha \in (1,2)$. For each $f_i$, we assume that \eqref{finfty} holds with $\alpha$ and 
\begin{equation}
 \exists C >0,\ \forall r \in [0,+\infty)^{m},\ \sum\limits_{i=1}^{m} f_i(r) \leq -C \left(\sum\limits_{i=1}^{m} r_i\right) \log^{\alpha}\left(2+\left(\sum\limits_{i=1}^{m} r_i\right)\right).
\label{AssCorGlobalSyst}
\end{equation}
Then, there exists $T$ sufficiently large such that \eqref{Semilinear} is globally null-controllable in time $T$.
\end{theo}
\begin{app}
Let $\alpha \in (1,2)$. There exists $T>0$ such that the system
$$
\left\{
\begin{array}{l l}
\partial_t u -  \Delta u =-u\log^{\alpha}(2+|u|+|v|)+ h_1 1_{\omega}  &\mathrm{in}\ (0,T)\times\Omega,\\
\partial_t v -  \Delta v =-v\log^{\alpha}(2+|u|+|v|)+ h_2 1_{\omega}  &\mathrm{in}\ (0,T)\times\Omega,\\
\frac{\partial u}{\partial n} = \frac{\partial v}{\partial n} = 0 &\mathrm{on}\ (0,T)\times\partial\Omega,\\
(u,v)(0,.)=(u_0,v_0) &\mathrm{in}\  \Omega,
\end{array}
\right.
$$
is globally null-controllable in time $T>0$.
\end{app}
\begin{proof}
As the proof is very similar to that of \Cref{TheoNeg}, we omit the details.\\
\indent The first step consists in steering the initial data to a nonnegative state in time $T_1:=1$. This is possible thanks to \Cref{TheoNegSyst}. After that, we use the following comparison principle between $u$, the solution to 
\begin{equation*}
\forall 1 \leq i \leq m,\ \left\{
\begin{array}{l l}
\partial_t u_i - d_i \Delta u_{i} =f_i(u_1, \dots, u_m)  &\mathrm{in}\ (T_1,T_2)\times\Omega,\\
\frac{\partial u_i}{\partial n} = 0 &\mathrm{on}\ (T_1,T_2)\times\partial\Omega,\\
u_i(T_1,.)=u_{i,T_1} &\mathrm{in}\  \Omega,
\end{array}
\right.
\end{equation*}
and $v$, the nonnegative (global) solution to the ordinary differential system
\begin{equation}
\label{EDOSyst}
\forall 1 \leq i \leq m,\ \left\{
\begin{array}{l l}
\dot{v_i}(t)  =  - f_i(v(t))&\mathrm{in}\ (T_1,+\infty),\\
v_i(T_1)=\norme{u_{i,T_1}}_{L^{\infty}(\Omega)}+1&,
\end{array}
\right.
\end{equation}
that is to say
\begin{equation}
\label{CompareEDOSyst}
\forall i \in \{1, \dots, m\},\ \forall t \in  [T_1,T_2], \ \text{a.e.}\ x \in \Omega,\ 0 \leq u_i(t,x) \leq v_i(t).
\end{equation}
This comes from the quasi-monotone nondecreasing of $(f_i)_{1 \leq i \leq m}$ which is a consequence of \eqref{QuasiPos} (see \cite[Theorem 12.2.1]{WYW} or also \cite[Chapter 8, Theorem 3.1]{Pao}).\\
\indent Then, by using \eqref{AssCorGlobalSyst}, \eqref{EDOSyst}, \eqref{CompareEDOSyst} and the arguments of the step $2$ of the proof of \Cref{CorGlobal}, we readily get 
\begin{equation*}
\label{CompareEDOBisSyst}
\forall i \in \{1, \dots, m\},\ \text{a.e.}\ x \in \Omega,\ 0 \leq u_i(T_2,x) \leq \delta,
\end{equation*}
where $T_2$ is chosen sufficiently large and $\delta> 0$ is the radius of the ball of $L^{\infty}(\Omega)^m$ centered at $0$ where the local null-controllability of \eqref{Semilinear} holds in time $T=1$ (see for instance \cite[Theorem 1.1]{FCGBT} and the small $L^{\infty}$ perturbations method).\\
\indent Then, one can steer $u(T_2,.)$ to $0$ with an appropriate choice of the control.
\end{proof}
Another interesting problem could be to determine if \Cref{TheoNegSyst} and \Cref{CorGlobalSyst} can be generalized with fewer controls than equations in \eqref{Semilinear}. The usual strategy of Luz de Teresa to ‘eliminate controls’ in a linear parabolic system (see \cite{dTInsen} or \cite[Theorem 4.1]{AKBGBT}) seems to be difficult to implement because the Carleman inequality in $L^1$ (see \Cref{CarlL1}) only provide estimates on the function (and not on its partial derivatives in time and space).\\ 

\textbf{Acknowledgments.}
I would like to very much thank Karine Beauchard and Michel Pierre (Ecole Normale Supérieure de Rennes) for many fruitful, stimulating discussions, helpful advice. I am also grateful to Frédéric Marbach (Ecole Normale Supérieure de Rennes) for suggesting me to look at \cite[Lemma 9]{Coron-Op} and to Sylvain Ervedoza (Institut de Mathématiques de Toulouse) for reading a preliminary draft of this paper and for pointing me the two references \cite{Lale} and \cite{LiYau}.

\bibliographystyle{plain}
\small{\bibliography{bibliordnonlin}}

\begin{thebibliography}{10}

\bibitem{AKBD}
Farid Ammar~Khodja, Assia Benabdallah, and C{\'e}dric Dupaix.
\newblock Null-controllability of some reaction-diffusion systems with one
  control force.
\newblock {\em J. Math. Anal. Appl.}, 320(2):928--943, 2006.

\bibitem{AKBGBT}
Farid Ammar-Khodja, Assia Benabdallah, Manuel Gonz{\'a}lez-Burgos, and Luz
  de~Teresa.
\newblock Recent results on the controllability of linear coupled parabolic
  problems: a survey.
\newblock {\em Math. Control Relat. Fields}, 1(3):267--306, 2011.

\bibitem{AT}
Sebastian Anita and Daniel Tataru.
\newblock Null controllability for the dissipative semilinear heat equation.
\newblock {\em Appl. Math. Optim.}, 46(2-3):97--105, 2002.
\newblock Special issue dedicated to the memory of Jacques-Louis Lions.

\bibitem{B2}
Viorel Barbu.
\newblock Exact controllability of the superlinear heat equation.
\newblock {\em Appl. Math. Optim.}, 42(1):73--89, 2000.

\bibitem{B}
Viorel Barbu.
\newblock Local controllability of the phase field system.
\newblock {\em Nonlinear Anal.}, 50(3, Ser. A: Theory Methods):363--372, 2002.

\bibitem{Barbu-Book}
Viorel Barbu.
\newblock {\em Controllability and stabilization of parabolic equations},
  volume~90 of {\em Progress in Nonlinear Differential Equations and their
  Applications}.
\newblock Birkh\"{a}user/Springer, Cham, 2018.
\newblock Subseries in Control.

\bibitem{Blondel}
Vincent~D. Blondel and Alexandre Megretski, editors.
\newblock {\em Unsolved problems in mathematical systems and control theory}.
\newblock Princeton University Press, Princeton, NJ, 2004.

\bibitem{CaHa}
Thierry Cazenave and Alain Haraux.
\newblock {\em An introduction to semilinear evolution equations}, volume~13 of
  {\em Oxford Lecture Series in Mathematics and its Applications}.
\newblock The Clarendon Press, Oxford University Press, New York, 1998.
\newblock Translated from the 1990 French original by Yvan Martel and revised
  by the authors.

\bibitem{C}
Jean-Michel Coron.
\newblock {\em Control and nonlinearity}, volume 136 of {\em Mathematical
  Surveys and Monographs}.
\newblock American Mathematical Society, Providence, RI, 2007.

\bibitem{Coron-Op}
Jean-Michel Coron.
\newblock Some open problems on the control of nonlinear partial differential
  equations.
\newblock In {\em Perspectives in nonlinear partial differential equations},
  volume 446 of {\em Contemp. Math.}, pages 215--243. Amer. Math. Soc.,
  Providence, RI, 2007.

\bibitem{dTInsen}
Luz de~Teresa.
\newblock Insensitizing controls for a semilinear heat equation.
\newblock {\em Comm. Partial Differential Equations}, 25(1-2):39--72, 2000.

\bibitem{DHP}
Robert Denk, Matthias Hieber, and Jan Pr{\"u}ss.
\newblock Optimal {$L^p$}-{$L^q$}-estimates for parabolic boundary value
  problems with inhomogeneous data.
\newblock {\em Math. Z.}, 257(1):193--224, 2007.

\bibitem{DoFCGBZ}
Anna Doubova, Enrique Fern\'{a}ndez-Cara, Manuel Gonz\'{a}lez-Burgos, and
  Enrique Zuazua.
\newblock On the controllability of parabolic systems with a nonlinear term
  involving the state and the gradient.
\newblock {\em SIAM J. Control Optim.}, 41(3):798--819, 2002.

\bibitem{DZZ}
Thomas Duyckaerts, Xu~Zhang, and Enrique Zuazua.
\newblock On the optimality of the observability inequalities for parabolic and
  hyperbolic systems with potentials.
\newblock {\em Ann. Inst. H. Poincar\'e Anal. Non Lin\'eaire}, 25(1):1--41,
  2008.

\bibitem{E}
Lawrence~C. Evans.
\newblock {\em Partial differential equations}, volume~19 of {\em Graduate
  Studies in Mathematics}.
\newblock American Mathematical Society, Providence, RI, second edition, 2010.

\bibitem{FeTa}
Klemens {Fellner} and Bao~Quoc {Tang}.
\newblock {Global classical solutions to quadratic systems with mass
  conservation in arbitrary dimensions}.
\newblock {\em ArXiv e-prints: 1808.01315}, August 2018.

\bibitem{FeCa}
Enrique Fern\'{a}ndez-Cara.
\newblock Null controllability of the semilinear heat equation.
\newblock {\em ESAIM Control Optim. Calc. Var.}, 2:87--103, 1997.

\bibitem{FCGBT}
Enrique Fern{\'a}ndez-Cara, Manuel Gonz{\'a}lez-Burgos, and Luz de~Teresa.
\newblock Controllability of linear and semilinear non-diagonalizable parabolic
  systems.
\newblock {\em ESAIM Control Optim. Calc. Var.}, 21(4):1178--1204, 2015.

\bibitem{FCGBGP2}
Enrique Fern{\'a}ndez-Cara, Manuel Gonz{\'a}lez-Burgos, Sergio Guerrero, and
  Jean-Pierre Puel.
\newblock Exact controllability to the trajectories of the heat equation with
  {F}ourier boundary conditions: the semilinear case.
\newblock {\em ESAIM Control Optim. Calc. Var.}, 12(3):466--483, 2006.

\bibitem{FCGBGP}
Enrique Fern{\'a}ndez-Cara, Manuel Gonz{\'a}lez-Burgos, Sergio Guerrero, and
  Jean-Pierre Puel.
\newblock Null controllability of the heat equation with boundary {F}ourier
  conditions: the linear case.
\newblock {\em ESAIM Control Optim. Calc. Var.}, 12(3):442--465, 2006.

\bibitem{FCG}
Enrique Fern{\'a}ndez-Cara and Sergio Guerrero.
\newblock Global {C}arleman inequalities for parabolic systems and applications
  to controllability.
\newblock {\em SIAM J. Control Optim.}, 45(4):1399--1446 (electronic), 2006.

\bibitem{FCZ}
Enrique Fern{\'a}ndez-Cara and Enrique Zuazua.
\newblock Null and approximate controllability for weakly blowing up semilinear
  heat equations.
\newblock {\em Ann. Inst. H. Poincar\'e Anal. Non Lin\'eaire}, 17(5):583--616,
  2000.

\bibitem{FI}
Andrei~V. Fursikov and Oleg~Yu. Imanuvilov.
\newblock {\em Controllability of evolution equations}, volume~34 of {\em
  Lecture Notes Series}.
\newblock Seoul National University, Research Institute of Mathematics, Global
  Analysis Research Center, Seoul, 1996.

\bibitem{GavaReg}
Victor~A. Galaktionov and Juan~L. V\'{a}zquez.
\newblock Regional blow up in a semilinear heat equation with convergence to a
  {H}amilton-{J}acobi equation.
\newblock {\em SIAM J. Math. Anal.}, 24(5):1254--1276, 1993.

\bibitem{GavaHam}
Victor~A. Galaktionov and Juan~L. Vazquez.
\newblock Blow-up for quasilinear heat equations described by means of
  nonlinear {H}amilton-{J}acobi equations.
\newblock {\em J. Differential Equations}, 127(1):1--40, 1996.

\bibitem{GaVaSurv}
Victor~A. Galaktionov and Juan~L. V\'{a}zquez.
\newblock The problem of blow-up in nonlinear parabolic equations.
\newblock {\em Discrete Contin. Dyn. Syst.}, 8(2):399--433, 2002.
\newblock Current developments in partial differential equations (Temuco,
  1999).

\bibitem{Granas}
Andrzej Granas.
\newblock On the {L}eray-{S}chauder alternative.
\newblock {\em Topol. Methods Nonlinear Anal.}, 2(2):225--231, 1993.

\bibitem{LSU}
Olga~Aleksandrovna. Ladyzenskaja, Vsevolod~Alekseevich Solonnikov, and
  Nina~Nikolaevna Uralceva.
\newblock {\em Linear and quasilinear equations of parabolic type}.
\newblock Translated from the Russian by S. Smith. Translations of Mathematical
  Monographs, Vol. 23. American Mathematical Society, Providence, R.I., 1968.

\bibitem{Lale}
Camille {Laurent} and Matthieu {L{\'e}autaud}.
\newblock {Observability of the heat equation, geometric constants in control
  theory, and a conjecture of Luc Miller}.
\newblock {\em ArXiv e-prints: 1806.00969}, June 2018.

\bibitem{LB}
Kévin Le~Balc'h.
\newblock Controllability of a 4 x 4 quadratic reaction–diffusion system.
\newblock {\em Journal of Differential Equations}, 2018, In press,
  doi:10.1016/j.jde.2018.08.046.

\bibitem{LLR}
J{\'e}r{\^o}me Le~Rousseau and Gilles Lebeau.
\newblock On {C}arleman estimates for elliptic and parabolic operators.
  {A}pplications to unique continuation and control of parabolic equations.
\newblock {\em ESAIM Control Optim. Calc. Var.}, 18(3):712--747, 2012.

\bibitem{LiYau}
Peter Li and Shing-Tung Yau.
\newblock On the parabolic kernel of the {S}chr\"{o}dinger operator.
\newblock {\em Acta Math.}, 156(3-4):153--201, 1986.

\bibitem{LCMML}
Juan L{\'{\i}}maco, Marcondes Clark, Alexandro Marinho, Silvado~B. de~Menezes,
  and Aldo~T. Louredo.
\newblock Null controllability of some reaction-diffusion systems with only one
  control force in moving domains.
\newblock {\em Chin. Ann. Math. Ser. B}, 37(1):29--52, 2016.

\bibitem{LJL}
Jacques-Louis Lions.
\newblock Exact controllability, stabilization and perturbations for
  distributed systems.
\newblock {\em SIAM Rev.}, 30(1):1--68, 1988.

\bibitem{Pao}
C.~V. Pao.
\newblock {\em Nonlinear parabolic and elliptic equations}.
\newblock Plenum Press, New York, 1992.

\bibitem{P}
Michel Pierre.
\newblock Global existence in reaction-diffusion systems with control of mass:
  a survey.
\newblock {\em Milan J. Math.}, 78(2):417--455, 2010.

\bibitem{ProtterWein}
Murray~H. Protter and Hans~F. Weinberger.
\newblock {\em Maximum principles in differential equations}.
\newblock Prentice-Hall, Inc., Englewood Cliffs, N.J., 1967.

\bibitem{S}
Jacques Simon.
\newblock Compact sets in the space {$L^p(0,T;B)$}.
\newblock {\em Ann. Mat. Pura Appl. (4)}, 146:65--96, 1987.

\bibitem{SoP}
Philippe {Souplet}.
\newblock {Global existence for reaction-diffusion systems with dissipation of
  mass and quadratic growth}.
\newblock {\em Journal of Evolution Equations, In press, ArXiv
  e-prints:1804.05193}, April 2018.

\bibitem{WZ}
Gensheng Wang and Liang Zhang.
\newblock Exact local controllability of a one-control reaction-diffusion
  system.
\newblock {\em J. Optim. Theory Appl.}, 131(3):453--467, 2006.

\bibitem{WYW}
Zhuoqun Wu, Jingxue Yin, and Chunpeng Wang.
\newblock {\em Elliptic \& parabolic equations}.
\newblock World Scientific Publishing Co. Pte. Ltd., Hackensack, NJ, 2006.

\bibitem{Zu-wave}
Enrique Zuazua.
\newblock Exact controllability for semilinear wave equations in one space
  dimension.
\newblock {\em Ann. Inst. H. Poincar\'{e} Anal. Non Lin\'{e}aire},
  10(1):109--129, 1993.

\bibitem{Zu-Fi}
Enrique Zuazua.
\newblock Finite-dimensional null controllability for the semilinear heat
  equation.
\newblock {\em J. Math. Pures Appl. (9)}, 76(3):237--264, 1997.

\bibitem{Z-hand}
Enrique Zuazua.
\newblock Controllability and observability of partial differential equations:
  some results and open problems.
\newblock In {\em Handbook of differential equations: evolutionary equations.
  {V}ol. {III}}, Handb. Differ. Equ., pages 527--621. Elsevier/North-Holland,
  Amsterdam, 2007.

\end{thebibliography}

\end{document}